\documentclass[oneside,article]{memoir}
\usepackage{amsmath}         
\usepackage{amsthm,thmtools,thm-restate} 
\usepackage{enumitem}        
  
\usepackage{tikz-cd}

\usepackage{float} 
\usepackage{caption, subcaption} 

\usepackage[nobysame,alphabetic,initials]{amsrefs} 
\usepackage[hidelinks]{hyperref}                   
\usepackage[capitalize,nosort]{cleveref}           

\usepackage{amssymb}
\usepackage[mathscr]{euscript}
\usepackage{relsize}
\usepackage{stmaryrd}
\usepackage{stackengine}

\usepackage{indentfirst} 
\usepackage{multicol}

\usepackage[inline,nomargin]{fixme} 
  \fxsetup{targetlayout=color}

\usepackage{float} 

\isopage[12]

\setlrmargins{*}{*}{1}
\checkandfixthelayout

\counterwithout{section}{chapter}
\usepackage{appendix}

  \newcommand{\adj}{\dashv}

  \newcommand{\iso}{\mathrel{\cong}}

  \newcommand{\we}{\mathrel\sim}
  
  \newcommand{\ho}{\operatorname{ho}}

  \newcommand{\op}{{\mathord\mathrm{op}}}

  \renewcommand{\lim}{\operatorname{lim}}

  \newcommand{\colim}{\operatorname{colim}}

  \newcommand{\bd}{\partial}

  \newcommand{\id}[1][]{\operatorname{id}_{#1}}
  
  \newcommand{\simp}[1]{\mathord\Delta^{#1}}

  \newcommand{\uvar}{\mathord{\relbar}}

  \renewcommand{\tilde}{\widetilde}

  \renewcommand{\hat}{\widehat}

  \newcommand{\set}[2]{\left\{#1\mathrel{}\middle|\mathrel{}#2\right\}}

  \renewcommand{\emptyset}{\varnothing}

  \newcommand{\bdL}{\bd_L}
  \newcommand{\bdR}{\bd_R}
  \newcommand{\bdT}{\bd_T}
  \newcommand{\bdB}{\bd_B}

  \makeatletter
  \DeclareRobustCommand{\sqcdot}{\mathbin{\mathpalette\morphic@sqcdot\relax}}
\newcommand{\morphic@sqcdot}[2]{%
  \sbox\z@{$\m@th#1\centerdot$}%
  \ht\z@=.33333\ht\z@
  \vcenter{\box\z@}%
}
\makeatother
  \newcommand{\bk}{\sqcdot}

  \newcommand{\cat}[1]{\mathscr{#1}}

  \newcommand{\ncat}[1]{\mathsf{#1}}

  \newcommand{\C}{\cat{C}}

  \newcommand{\PsTop}{\ncat{PsTop}}

  \newcommand{\DTop}{\ncat{DTop}}

  \newcommand{\Frm}{\ncat{Frm}}

  \newcommand{\Sob}{\ncat{Sob}}

  \newcommand{\Loc}{\ncat{Loc}}

  \renewcommand{\op}[1]{\operatorname{#1}}

  \newcommand{\from}{\colon}

  \newcommand{\ito}{\hookrightarrow}

  \newcommand{\cto}{\rightarrowtail}

  \newcommand{\fto}{\twoheadrightarrow}

  \newcommand{\afto}{\mathrel{\ensurestackMath{\stackon[-1pt]{\fto}{\mathsmaller{\mathsmaller\we}}}}}

  \renewcommand{\iff}{if and only if}
  
  \newcommand{\wrt}{with respect to}

  \newcommand{\llp}{left lifting property}
  \newcommand{\rlp}{right lifting property}

  \declaretheorem[style=definition,within=section]{definition}
    \declaretheorem[style=definition,within=section]{notation}

  \declaretheorem[style=plain,numberlike=definition]{corollary}
  \declaretheorem[style=plain,numberlike=definition]{lemma}
  \declaretheorem[style=plain,numberlike=definition]{proposition}
  \declaretheorem[style=plain,numberlike=definition]{theorem}

  \declaretheorem[style=plain,numbered=no,name=Theorem]{theorem*}

 \Crefname{notation}{Notation}{Notations}
  \Crefname{corollary}{Corollary}{Corollaries}
  \Crefname{definition}{Definition}{Definitions}
  \Crefname{lemma}{Lemma}{Lemmas}
  \Crefname{proposition}{Proposition}{Propositions}
  \Crefname{remark}{Remark}{Remarks}
  \Crefname{theorem}{Theorem}{Theorems}

  \newlist{axioms}{enumerate}{1}
  \Crefname{axioms}{}{}

  \newenvironment{tikzeq*}
  {
    \begingroup
    \begin{equation*}
    \begin{tikzpicture}[baseline=(current bounding box.center)]
  }
  {
    \end{tikzpicture}
    \end{equation*}
    \endgroup
    \ignorespacesafterend
  }

  \tikzset
  {
    diagram/.style=
    {
      matrix of math nodes,
      column sep={4.3em,between origins},
      row sep={4em,between origins},
      text height=1.5ex,
      text depth=.25ex
    },
    over/.style={preaction={draw=white,-,line width=6pt}},
    every to/.style={font=\footnotesize},
    inj/.style={right hook->},
    surj/.style={-{Latex[open]}},
    cof/.style={>->},
    fib/.style={->>},
  }

  \DeclareFontFamily{U}{mathx}{\hyphenchar\font45}

  \DeclareFontShape{U}{mathx}{m}{n}{
    <5> <6> <7> <8> <9> <10>
    <10.95> <12> <14.4> <17.28> <20.74> <24.88>
    mathx10}{}

  \DeclareSymbolFont{mathx}{U}{mathx}{m}{n}

  \DeclareFontFamily{U}{mathb}{\hyphenchar\font45}

  \DeclareFontShape{U}{mathb}{m}{n}{
    <5> <6> <7> <8> <9> <10>
    <10.95> <12> <14.4> <17.28> <20.74> <24.88>
    mathb10}{}

  \DeclareSymbolFont{mathb}{U}{mathb}{m}{n}

  \DeclareMathAccent{\widebar}{0}{mathx}{"73}

  \DeclareMathSymbol{\Rsh}{\mathrel}{mathb}{"E9}

  \DeclareFontFamily{U}{MnSymbolA}{}

  \DeclareFontShape{U}{MnSymbolA}{m}{n}{
    <-6> MnSymbolA5
    <6-7> MnSymbolA6
    <7-8> MnSymbolA7
    <8-9> MnSymbolA8
    <9-10> MnSymbolA9
    <10-12> MnSymbolA10
    <12-> MnSymbolA12}{}

  \DeclareSymbolFont{MnSyA}{U}{MnSymbolA}{m}{n}

  \DeclareMathSymbol{\twoheaddownarrow}{\mathrel}{MnSyA}{27}

  \newcommand{\MSC}[1]{%
    \let\thempfn\relax
    \footnotetext[0]{2020 Mathematics Subject Classification: #1.}
  }

\newcommand{\Set}{\mathsf{Set}}
\newcommand{\Top}{\ncat{Top}}

\newcommand{\dom}{\operatorname{dom}} 

\newcommand{\adh}{\operatorname{adh}}
\newcommand{\const}{\operatorname{const}}
\newcommand{\inc}{\operatorname{inc}}
\newcommand{\pr}{\operatorname{pr}}

\title{Synthetic approach to the Quillen model structure on topological spaces}
\author{Sterling Ebel \and Krzysztof Kapulkin}
\date{\today}

\begin{document}

\maketitle

\begin{abstract}
    We provide an axiomatic treatment of Quillen's construction of the model structure on topological spaces to make it applicable to a wider range of settings, including $\Delta$-generated spaces and pseudotopological spaces.
    We use this axiomatization to construct a model structure on the category of locales.
\end{abstract}

\section*{Introduction}

The construction of Quillen's model structure on the category $\Top$ of topological spaces is a cornerstone result of modern homotopy theory \cite[{\S}II.3]{quillen:homotopical-algebra}.
The importance of many other model structures, like the Kan--Quillen model structure on simplicial sets or the Thomason model structure on the category of small categories \cite{thomason}, is justified by their equivalence with Quillen's structure on topological spaces.
Numerous accounts of it have been given since, including \cite{hovey:book}, \cite{hirschhorn:book}, \cite[\S17.2]{may-ponto:more-concise}, and \cite{hirschhorn:model-structure-on-top}, among others, further underscoring its significance.

Although the category of (all) topological spaces is perhaps most commonly studied, quite often one works with either subcategories of topological spaces (e.g., compactly generated weakly Hausdorff spaces or $\Delta$-generated spaces) or larger categories containing $\Top$ as a subcategory (e.g., pseudotopological spaces).
The former are natural choices for a convenient category of spaces, i.e., a subcategory of $\Top$ with nice categorical properties, including cartesian closure.
The latter, being a common generalization of both graphs and spaces, finds applications in topological data analysis when quantifying to what extent a space can be recovered by sampling only finitely many of its points \cite{rieser:pseudotopology,rieser:cech}.

The purpose of this paper is to axiomatize the requirements on various categories of interest that are required to make a more modern version of Quillen's argument work.
Such a modern treatment differs from Quillen's argument in a variety of ways, including: isolating the pushout-product construction, utilizing characterization of the class of weak equivalences in terms of an ``up to homotopy'' lifting property, and identifying the key role of the subcategory of topological cubes. 
Our axiomatization first introduces the notion of a \emph{category with intervals} (\cref{S-axioms}), a framework in which one can speak of homotopies, defined using the topological interval, and their basic properties.
We then introduce the framework of a \emph{Q-structure} (\cref{Q-axioms}), which is the structure on a category that permits Quillen's construction (\cref{SerreModelStructure}).
Moreover, in a \emph{good} Q-structure, weak equivalences can be defined in terms of isomorphisms on the sets of connected components and all higher homotopy groups (\cref{WEClassicDef}).

Several examples of good Q-structures are identified and give rise to model structures on: topological spaces (recovering the model structure originally constructed in \cite{quillen:homotopical-algebra}), compactly generated weakly Hausdorff spaces (recovering the model structure of \cite{hovey:book}), $\Delta$-generated spaces (recovering the model structure of \cite{haraguchi}), sober spaces (\cref{Ex:topological_spaces}), pseudotopological spaces (\cref{ex:pseudotop}, recovering the model structure of \cite{rieser:pseudotopology}), and locales (\cref{ex:locales}).
This in particular settles an open question of constructing a model structure on the category of locales.

This paper is organized as follows.
We begin in \cref{BasicNotions} by introducing the notion of a category with intervals (\cref{S-axioms}) and developing basic homotopy theory therein.
We then proceed to introduce (good) Q-structures in \cref{sec:q-structure} and define cofibrations and fibrations.
We conclude our proof of the existence of a model structure in \cref{SerreSection} before turning our attention to examples of Q-structures in \cref{Ex:topological_spaces,ex:pseudotop,ex:locales}.

\textbf{Acknowledgement.}
We thank Karol Szumi{\l}o for many helpful conversation about the Quillen model structure on topological spaces.
We thank the anonymous referee for several comments that greatly improved the presentation of these results.

The work of S.E.~was supported by an Undergraduate Student Research Award from the Natural Sciences and Engineering Research Council.
This material is based upon work supported by the National Science Foundation under Grant No.~DMS-1928930 while K.K.~participated in a program hosted by the Mathematical Sciences Research Institute in Berkeley, California, during the 2022–23 academic year.

\section{Basic notions in a nice category of spaces}\label{BasicNotions}

In this section, we introduce the notion of a category with intervals (\cref{S-axioms}) and show that it is sufficient to develop basics of the theory of homotopies.

Throughout the paper, we write $\Top$ for the category of (all) topological spaces and continuous maps.
For notational convenience, the unit interval $[0,1]$ will be denoted $I$.

Let $\C$ denote a bicomplete category, to be thought of as a ``nice" category of spaces. As such, we will refer to objects in $\C$ as spaces, and morphisms in $\C$ as (continuous) maps. Let $\emptyset$ and $*$ denote the initial and terminal object of $\C$, respectively. Let $\Box$ be the full subcategory of $\ncat{Top}$ consisting of all $n$-cube and their boundaries, for $n \geq 0$. Let $\Box_{\leq 2}$ denote the full subcategory of $\Box$ whose objects are $*$, $I$, $I^2$, $\emptyset$, $\bd I$, $\bd I^2$.

Let $e_k \from * \to I$ be the the appropriate endpoint inclusion for $k \in \{0,1\}$. Denote the left, top, right, and bottom edge inclusions $I \ito I^2$ by $\bdL = e_0 \times \id[I]$, $\bdT=\id[I] \times e_1$, $\bdR = e_1 \times \id[I]$, and $\bdB = \id[I] \times e_0$ respectively. By embedding, we shall mean a faithful functor throughout this paper.

\begin{definition}\label{S-axioms}
A \textit{category with intervals}\footnote{Note that \cite{morel-voevodsky} introduces the notion of a \emph{site with interval}. Our notion differs from theirs, as it is strongly based on the properties of the topological interval, whereas that of \cite{morel-voevodsky} is abstract and specific to the category of sheaves.} is a bicomplete category $\C$ along with an embedding $\iota \from \Box_{\leq 2} \ito \C$ such that:
\begin{enumerate}
\item[(S1)] $\iota\emptyset$ is initial and $\iota*$ is terminal;
\item[(S2)] $\iota I^2 \iso \iota I \times \iota I$, where the product is taken in $\C$;
\item[(S3)] $\iota$ preserves a pushout of the form
\[\begin{tikzcd}
	{*} & I \\
	I & I
	\arrow["{e_0}", from=1-1, to=1-2]
	\arrow["{e_1}"', from=1-1, to=2-1]
	\arrow["a"', from=2-1, to=2-2]
	\arrow["\lrcorner"{anchor=center, pos=0.125, rotate=180}, draw=none, from=2-2, to=1-1]
	\arrow["b", from=1-2, to=2-2]
\end{tikzcd}\]
such that $a(0)=0$ and $b(1)=1$, and $X \times \uvar$ preserves this pushout for all $X \in \C$;
\item[(S4)] $\emptyset \times X \iso \emptyset$ for all $X \in \C$.
\end{enumerate}
\end{definition}

The axioms given above are fairly rigid in structure; as such, it is difficult to give a ``non-canonical'' example. We give several examples in \cref{Ex:topological_spaces,ex:pseudotop,ex:locales}, which will moreover be of good Q-structure (cf. \cref{Q-axioms}). All of these examples are given by restricting a functor $\Top \to \C$ to $\Box$. We are unaware of any examples of categories which admit two distinct embeddings of $\Box$ defining two distinct categories with intervals.

We require that the pushout in (S3) be of maps in $\Box_{\leq 2}$ to use relationships internal to $\Box_{\leq 2}$ to develop a basic theory of homotopies. 
Necessarily, $a$ and $b$ are injective, hence homeomorphisms onto their image. Unless the distinction is necessary, we will identify elements in $\Box_{\leq 2}$ with their image in $\C$. For the rest of this section, fix an embedding $\iota$ making $(\C,\iota)$ into a category with intervals.

\subsection*{Homotopies}

In this subsection, working in an arbitrary category with intervals (\cref{S-axioms}), we develop a notion of homotopy, leading to the proof (\cref{Bicategory}) that any such category carries a natural structure of a 2-category.

\begin{definition}\label{HomotopyDef}
Let $X$ and $Y$ be spaces and $f,g \from X \to Y$ be maps. A \textit{homotopy} from $f$ to $g$ is given by a map $H \from X \times I \to Y$ such that the following commutes:
\[\begin{tikzcd}
	X & {X \times I} & X \\
	& Y
	\arrow["{e_1}"', from=1-3, to=1-2]
	\arrow["{e_0}", from=1-1, to=1-2]
	\arrow["H"{description}, from=1-2, to=2-2]
	\arrow["f"', from=1-1, to=2-2]
	\arrow["g", from=1-3, to=2-2]
\end{tikzcd}\]
If there exists a homotopy from $f$ to $g$ then $f$ \textit{is homotopic to} $g$, denoted by $f \sim g$ or $H \from f \sim g$.
\end{definition}

In particular, we will be using $X \times I$ as a cylinder object for $X$, in which case our notion of homotopy is that of a \textit{left homotopy}. In \cref{sec:q-structure}, we will assume further that $\uvar \times I$ admits a right adjoint $(\uvar)^I$, giving path spaces $X^I$; this gives an equivalent notion of \textit{right homotopy}, which we will not explicitly use.

Let $f,g \from X \to Y$, $u \from W \to X$, and $v \from Y \to X$ be maps, and $H$ a homotopy from $f$ to $g$. For a space $A$, let $\pi_A \from A \times I \to A$ denote the projection. We will follow these notational conventions for homotopies:
\begin{enumerate}
\item[(1)] $Hu = H(u \times \id[I])$, which gives a homotopy from $fu$ to $gu$;
\item[(2)] $\const_f=f\pi_X=\pi_Y(f \times \id[I])$.
\end{enumerate}
Homotopies are subject to the following identities, which can easily be verified:
\begin{enumerate}
\item[(1)] $(u \times \id[I])e_k=e_k u$;
\item[(2)] $\pi_X e_k=\id[X]$;
\item[(3)] $v \circ \const_f \circ (u \times \id[X]) = \const_{vfu}$.
\end{enumerate}

\begin{definition}
Let $f,g \from X \to Y$ be maps, and suppose $i \from A \to X$ is a map such that $fi=gi$. A \textit{homotopy relative to }$A$ is a homotopy $H \from X \times I \to Y$ from $f$ to $g$ such that the following commutes:
\begin{center}
\begin{tikzcd}
	{A \times I} & A \\
	{X \times I} & Y
	\arrow["{i \times \id[I]}"', hook, from=1-1, to=2-1]
	\arrow["H"', from=2-1, to=2-2]
	\arrow["{fi=gi}", from=1-2, to=2-2]
	\arrow["{\pi_A}", from=1-1, to=1-2] 
\end{tikzcd}
\end{center}
Explicitly, $Hi=\const_{fi}$.
If such a homotopy exists, it is denoted $f \sim g$ rel $A$.
\end{definition}

\begin{definition}
Let $f,g,h \from X \to Y$ be maps. Given a homotopy $H$ from $f$ to $g$ and a homotopy $K$ from $g$ to $h$, their \textit{track composite} $H \bk K$ is the induced map in the pushout
\[\begin{tikzcd}
	X && {X \times I} \\
	{X \times I} && {X \times I} \\
	&&& Y
	\arrow["{e_0}", from=1-1, to=1-3]
	\arrow["{e_1}"', from=1-1, to=2-1]
	\arrow["{\id[X] \times a}"', from=2-1, to=2-3]
	\arrow["\lrcorner"{anchor=center, pos=0.125, rotate=180}, draw=none, from=2-3, to=1-1]
	\arrow["{\id[X] \times b}", from=1-3, to=2-3]
	\arrow["H"', bend right, from=2-1, to=3-4]
	\arrow["K", bend left, from=1-3, to=3-4]
	\arrow["{H\bk K}"{description}, dashed, from=2-3, to=3-4]
\end{tikzcd}\]
Let $-\id[I] \from I \to I$ be the map $t \mapsto 1-t$. The \textit{inverse homotopy} of $H$, denoted $-H$, is $H(\id[X] \times (-\id[I]))$:
\[\begin{tikzcd}
	{X \times I} && {X \times I} \\
	&& Y
	\arrow["{\id[X] \times (-\id[I])}", from=1-1, to=1-3]
	\arrow["H", from=1-3, to=2-3]
	\arrow["{-H}"', from=1-1, to=2-3]
\end{tikzcd}\]
\end{definition}

Since $(\id[X] \times a)e_0=e_0$ and $(\id[X] \times b)e_1=e_1$ by (S3), $H \bk K$ is a homotopy from $f$ to $h$. Similarly, $-H$ gives a homotopy from $g$ to $f$.
Given maps $q \from W \to X$ and $r \from Y \to Z$, we have $q(H \bk K)r = qHr \bk qKr$.
Thus, if $fi=gi=hi$ for $i \from A \to X$, then $(H \bk K)i=Hi \bk Ki = \const_{fi}$, and similarly $-Hi=\const_{fi}$, showing that composition and inverse homotopies preserve relativity.

\begin{definition}
Let $H,K \from X \times I \to Y$ be homotopies from $f$ to $g$. Then $H$ and $K$ are \textit{homotopic rel endpoints} if there is a map $\alpha \from X \times I^2 \to Y$ such that $\alpha \bdL = H$, $\alpha \bdT = \const_g$, $\alpha \bdR = K$, and $\alpha \bdB=\const_f$.
\end{definition}

Analogous comments to the above hold for homotopies rel endpoints. Thus, constant homotopies, inverse homotopies, and track composites show respectively that (relative) homotopy is reflexive, symmetric, and transitive. In particular, we obtain the following:

\begin{corollary}
Homotopy defines an equivalence relation on $\C(X,Y)$ for every pair of objects $X,Y \in \C$. Similarly, for any fixed pair of maps $f,g \from X \to Y$, homotopy rel endpoints defines an equivalence relation on the set of homotopies from $f$ to $g$. \qed
\end{corollary}

\begin{lemma}\label{HomotopyCompWellDefined}
Let $H$ and $H'$ be homotopies from $f$ to $g$, and $K$ and $K'$ be homotopies from $g$ to $h$. If $H \sim H'$ rel endpoints and $K \sim K'$ rel endpoints, then $H\bk K \sim H'\bk K'$ rel endpoints. That is, track composition is well-defined up to homotopy rel endpoints.
\end{lemma}
\begin{proof}
By the assumptions, we may choose a map $\alpha \from X \times I^2 \to Y$ such that $\alpha \bdL=H$, $\alpha \bdT = \const_g$, $\alpha \bdR= H'$, and $\alpha \bdB = \const_f$; likewise, we may choose $\beta \from X \times I^2 \to Y$ with $\beta \bdL=K$, $\beta \bdT = \const_h$, $\beta \bdR=K'$, and $\beta \bdB=\const_g$. Then the induced map in the following pushout
\[\begin{tikzcd}
	{X \times I} && {X \times I \times I} \\
	{X \times I \times I} && {X \times I \times I} \\
	&&& Y
 	\arrow["{\id[X] \times \bdT}"', from=1-1, to=2-1]
	\arrow["{\id[X] \times \bdB}", from=1-1, to=1-3]
	\arrow[from=2-1, to=2-3]
	\arrow[from=1-3, to=2-3]
	\arrow["{\zeta}"{description}, dashed, from=2-3, to=3-4]
	\arrow["\beta", bend left, from=1-3, to=3-4]
	\arrow["\alpha"', bend right=15, from=2-1, to=3-4]
\end{tikzcd}\]
gives a homotopy rel endpoints from $H \bk K$ to $H' \bk K'$. 
\end{proof}

\begin{lemma}\label{ConstantHomotopyIdentity}
For any homotopy $H$ from $f$ to $g$, $H \bk \const_g \sim H$ rel endpoints and $\const_f \bk H \sim H$ rel endpoints.
\end{lemma}
\begin{proof}
Let $\Gamma \from I^2 \to I$ be given by $\Gamma(s,t)= st+(1-s)a^{-1}(\min(t,a(1)))$. Then the map $\alpha = H\Gamma$ gives the required homotopy. To see this, note that $\Gamma(e_0 \times a)=\id[I]$ and $\Gamma(e_0 \times b) = e_1$, so $\alpha \bdL = H \bk \const_g$, and clearly $\alpha \bdR = H$, $\alpha \bdB = \const_f$, and $\alpha \bdT = \const_g$.
\end{proof}

\begin{lemma}\label{HomotopyInverseComposition}
For any homotopy $H$ from $f$ to $g$, $H\bk (-H) \sim \const_f$ rel endpoints and $-H \bk H \sim \const_g$ rel endpoints.
\end{lemma}
\begin{proof}
Let $\gamma \from I \to I$ be given by $\gamma(t)=a^{-1}(t)$ for $t \leq a(1)$ and $\gamma(t) = 1-b^{-1}(t)$ for $t \geq b(0)$. Define $\Gamma \from I^2 \to I$ by
$$
\Gamma(s,t)=\begin{cases}
\min(a^{-1}(t), 1-s) & :\ t \leq u(1)\text{;}\\
\min(1-b^{-1}(t), 1-s) & :\ t \geq v(0)\text{.}
\end{cases}
$$
Then $\Gamma$ is continuous by the pasting lemma, $\Gamma \bdL = \gamma$, and $\Gamma \bdT= \Gamma \bdB = \Gamma \bdR = 0$.
Note that $\gamma a = \id[I]$ and $\gamma b = -\id[I]$, so $H(\id[X] \times \gamma a)=H$ and $H(\id[X] \times \gamma b)=-H$, hence $H(\id[X] \times \gamma)=H \bk (-H)$. Thus, $H(\id[X] \times \Gamma) \from X \times I^2 \to Y$ gives a homotopy from $H \bk(-H)$ to $\const_f$ rel endpoints. The other part is analogous.
\end{proof}

\begin{lemma}\label{HomotopyCompAssociative}
For any homotopies $H$ from $f_1$ to $f_2$, $J$ from $f_2$ to $f_3$, and $K$ from $f_3$ to $f_4$. Then $(H \bk J) \bk K \sim H \bk (J \bk K)$ rel endpoints. That is, track composition of homotopies is associative up to homotopy rel endpoints.
\end{lemma}
\begin{proof}
Let $\Omega = a(1)=b(0)$. As before, note that $a^{-1}$ is defined and continuous on $[0,\Omega]$, as is $b^{-1}$ on $[\Omega, 1]$. Let $\gamma \from I \to I$ be given by
$$
\gamma(t) = \begin{cases}
b^{-1}(t)         & \text{if } t \in [b\Omega,1]\text{;}\\
aba^{-1}b^{-1}(t) & \text{if } t \in [\Omega, b\Omega]\text{;}\\
a(t)              & \text{if } t \in [0, \Omega]\text{.}
\end{cases}
$$
By the pasting lemma, $\gamma$ is continuous, and $\gamma a = aa$, $\gamma ba = ab$, and $\gamma bb = b$. Let $\Gamma \from I^2 \to I$ be defined by $\Gamma(s,t)=(1-s)t+s\gamma(t)$. Then $\Gamma \bdB=e_0$, $\Gamma \bdT= e_1$, $\Gamma \bdL= \id[I]$, and $\Gamma \bdR = \gamma$. Let $\alpha = ((H \bk J) \bk K) \circ (\id[X]\times \Gamma)$. Clearly $\alpha \bdL= (H \bk J) \bk K$, $\alpha \bdB = \const_{f_1}$, and $\alpha\bdT = \const_{f_4}$. Moreover, $(\alpha \bdR)a= ((H \bk J) \bk K)aa=H$, and similarly $(\alpha \bdR)ba=J$ and $(\alpha \bdR)bb=K$. Thus, both $(\alpha \bdR)$ and $H \bk (J \bk K)$ fit in the following diagram, so by uniqueness they are equal, giving the required homotopy.
\[\begin{tikzcd}
	X & {X \times I} \\
	{X \times I} & {X \times I} \\
	&& Y
	\arrow[dashed, from=2-2, to=3-3]
	\arrow["b", from=1-2, to=2-2]
	\arrow["a"', from=2-1, to=2-2]
	\arrow["{e_0}", from=1-1, to=1-2]
	\arrow["{e_1}"', from=1-1, to=2-1]
	\arrow["H"', bend right, from=2-1, to=3-3]
	\arrow["{J \bk K}", bend left, from=1-2, to=3-3]
	\arrow["\lrcorner"{anchor=center, pos=0.125, rotate=180}, draw=none, from=2-2, to=1-1]
\end{tikzcd}\]
\end{proof}

\begin{lemma}\label{SquareComposition}
Let $f,g_0,g_1,h \from X \to Y$ be maps and $A \from f \sim g_0$, $B \from g_0 \sim h$, $C \from f \sim g_1$, and $D \from g_1 \sim h$ be homotopies. If there is a map $\alpha \from X \times I^2 \to Y$ with $\alpha \bdB=A$, $\alpha\bdR=B$, $\alpha \bdL=C$, and $\alpha \bdT=D$, then $A \bk B \sim C \bk D$ rel endpoints.
\end{lemma}
\begin{proof}
Let $\gamma_0,\gamma_1 \from I \to I^2$ be given by
$$
\gamma_0(x)=\begin{cases}
(a^{-1}(x),0) & \text{if } x \leq a(1)\\
(1,b^{-1}(x)) & \text{if } x \geq b(0)
\end{cases} \qquad \text{and} \qquad \gamma_1(x)=\begin{cases}
(0,a^{-1}(x)) & \text{if } x \leq a(1)\\
(b^{-1}(x),1) & \text{if } x \geq b(0)
\end{cases}
$$
Then $\gamma_0 a = \bdB$, $\gamma_0 b = \bdR$, $\gamma_1 a = \bdL$, and $\gamma_1 b = \bdT$. Let $\Gamma \from I^2 \to I^2$ be given by $\Gamma(s,t)=s\gamma_1(t)+(1-s)\gamma_0(t)$. Since $\Gamma\bdB=e_0 \times e_0$ and $\Gamma\bdT=e_1 \times e_1$, $\alpha(\id[X] \times \Gamma)$ gives a homotopy rel endpoints from $A \bk B$ to $C \bk D$.
\end{proof}

\begin{lemma}\cite{vogt:note}*{Lemma 1}\label{HtpyCompWellDefined}
Let $u, v \from X \to Y$ and $f,g \from Y \to Z$ be maps. Given a homotopy $H$ from $u$ to $v$ and a homotopy $K$ from $f$ to $g$, $fH \bk Kv \sim Ku \bk gH$ rel endpoints.
\end{lemma}

\begin{proof}
Let $\alpha \from X \times I \times I \to Z$ be $K (H \times \id[I])$. Then $\alpha \bdL=K(H \times \id[I])(e_0 \times \id[I]) = K(He_0) = Ku$, and similarly $\alpha \bdT= gH$, $\alpha\bdR=Kv$, and  $\alpha \bdB=fH$. The result follows from \cref{SquareComposition}.
\end{proof}

\begin{theorem}\label{Bicategory}
If $\C$ satisfies the assumptions in \cref{S-axioms}, then $\C$ admits a 2-category structure, where each pair of spaces $A$ and $B$ in $\C$ are assigned the groupoid $\C(A,B)$, whose objects are maps $A \to B$ and morphisms are homotopy classes of homotopies rel endpoints.
\end{theorem}
\begin{proof}
The composition in $\C(A,B)$ is given by track composition as in \cref{HomotopyDef}. \cref{HomotopyCompWellDefined,ConstantHomotopyIdentity,HomotopyInverseComposition,HomotopyCompAssociative} assert that this is well-defined and associative, that each map $f \from A \to B$ has an identity homotopy $\const_f$, and then each class of homotopies $[H]$ has an inverse $[-H]$. Thus, $\C(A,B)$ is a groupoid for every pair of spaces $A$ and $B$. For each space $A$, the functor $I_A \from [0] \to \C(A,A)$ maps the unique morphism in $[0]$ to $[\const_{\id[A]} \from \id[A] \sim \id[A]]$. The composition functors $c_{A,B,C} \from \C(A,B) \times \C(B,C) \to \C(A,C)$ act on objects by sending a pair of maps $(f \from A \to B,g \from B \to C)$ to $gf \from A \to C$, and by \cref{HtpyCompWellDefined}, there is a canonical, well-defined choice for a pair of homotopy classes $([H\from u \sim v],[K \from f \sim g])$, namely mapping them to $[fH \bk Kv]= [Ku \bk gH]$.
\end{proof}

We now turn our attention to defining homotopy equivalences.

\begin{definition} 
A map $f \from X \to Y$ is a \textit{homotopy equivalence} if there exists $g \from Y \to X$ such that $gf \sim \id[X]$ and $fg \sim \id[Y]$. If $gf = \id[X]$ and $fg \sim \id[Y]$ rel $X$, then $g$ is a \textit{deformation retraction}.
\end{definition}

\begin{definition}\label{ClSphDef} 
Let $\ho\C$ denote the category whose objects are spaces in $\C$ and whose maps are homotopy classes of continuous maps. By \cref{HtpyCompWellDefined}, composition of homotopy classes is well defined. An isomorphism in this category is a class of homotopy equivalences. There is a canonical quotient functor $\Pi \from \C \to \ho\C$, which maps $f \from X \to Y$ to $[f] \from X \to Y$.
\end{definition}

We record two lemmas about homotopy equivalences for future use.

\begin{lemma}\label{HEReflectionLemma}
A map $f \from X \to Y$ is a homotopy equivalence \iff\ $\Pi f$ is an isomorphism in $\ho\C$.
\end{lemma}
\begin{proof}
If $f$ has a homotopy inverse $g \from Y \to X$, then since $gf \sim \id[X]$ and $fg \sim \id[Y]$, we have $[g] [f] = [\id[X]]$ and $[f] [g] = [\id[Y]]$ so $\Pi f = [f]$ is an isomorphism. Conversely, if $\Pi f = [f]$ is an isomorphism, choose a representative $g \from Y \to X$ for $[f]^{-1}$. Then $[gf]=[g][f] = [\id[X]]$, so $gf \sim \id[X]$ and likewise $fg \sim \id[Y]$ so $f$ is a homotopy equivalence.
\end{proof}

\begin{lemma}\label{HEProperties}
The following properties hold for homotopy equivalences:
\begin{itemize}
\item[(1)] Homotopy equivalences are closed under 2-out-of-6.
\item[(2)] A retract of a homotopy equivalence in the arrow category is again a homotopy equivalence.
\item[(3)] If $f \from X \to Y$ and $f' \from X' \to Y'$ are homotopy equivalences, then so is $f \times f'$.
\item[(4)] Assuming $\uvar \times I$ preserves coproducts, then if $\{f_k\}_{k \in K}$ is a family of homotopy equivalences, so is $\coprod_{k\in K} f_k$.
\end{itemize}
\end{lemma}
\begin{proof}
(1) and (2) are immediate from \cref{HEReflectionLemma}, since isomorphisms satisfy these in any category and homotopy equivalences are precisely the maps inverted by $\Pi$.\\
For (3), choose homotopy inverses $g \from Y \to X$ and $g' \from Y' \to X'$, as well as homotopies $H$ from $gf$ to $\id[X]$ and $H'$ from $g'f'$ for $\id[X']$. Then a homotopy from $gf \times g'f' = (g \times g')(f \times f')$ to $\id[X \times X']$ is given by $(H,H')$. Similar reasoning gives a homotopy from $(f \times f')(g \times g')$ to $\id[Y \times Y']$, so $f \times f'$ is a homotopy equivalence.\\
Lastly for (4), given families $\{g_k\}_{k \in K}$ of homotopy inverses, $\{H_k\}_{k\in K}$ and $\{J_k\}_{k\in K}$ homotopies from $g_kf_k$ and $f_kg_k$ to the respective identities, $\coprod H_k$ and $\coprod J_k$ give homotopies from $\coprod g_k \circ \coprod f_k$ and $\coprod f_k \circ \coprod g_k$ to the respective identities since $(\coprod \dom(f_k)) \times I \iso \coprod (\dom(f_k) \times I)$.
\end{proof}

\subsection*{Pushout-products}

In this final subsection, we collect the requisite background on the Leibniz construction.
Specifically, we recall the definitions of pushout product (\cref{def:pushout-product}) and pullback power (\cref{def:pullback-power}) and prove their basic properties in the category of topological spaces.

\begin{definition} \label{def:pushout-product}
Let $f \from X \to Y$ and $g \from A \to B$ be maps in $\C$. Their \textit{pushout product}, denoted $f \hat{\times} g$, is the factorization of $f \times g \from X \times A \to Y \times B$ through the pushout in the following diagram:
\[\begin{tikzcd}
	{X \times A} && {Y \times A} \\
	{X \times B} && P \\
	&&& {Y \times B}
	\arrow["{f \times \id[A]}", from=1-1, to=1-3]
	\arrow["{\id[X] \times g}"', from=1-1, to=2-1]
	\arrow["{\id[X] \times g}"{description}, from=1-3, to=3-4, bend left]
	\arrow["{f \times \id[A]}"{description}, from=2-1, to=3-4, bend right]
	\arrow["{f \hat{\times} g}"{description}, dashed, from=2-3, to=3-4]
	\arrow[from=2-1, to=2-3]
	\arrow[from=1-3, to=2-3]
	\arrow["\lrcorner"{anchor=center, pos=0.125, rotate=180}, draw=none, from=2-3, to=1-1]
\end{tikzcd}\]
Note that since $\times$ is symmetric, $f \hat{\times} g \iso g \hat{\times} f$, and that $\hat{\times}$ is associative up to isomorphism.
\end{definition}

\begin{lemma}\label{PushProdArrowCategory}
If $f \from A \to B$ and $f' \from A' \to B'$ are isomorphic in the arrow category, as are $g \from C \to D$ and $g' \from C' \to D'$, then $f \hat{\times} g$ and $f' \hat{\times} g'$ are isomorphic in the arrow category.
\end{lemma}
\begin{proof}
This is a standard diagram chase, given that naturally isomorphic diagrams induce an isomorphism between their colimits which commutes with the colimit legs.
\end{proof}

A functor $\otimes \from \cat{A} \times \cat{B} \to \cat{C}$ is \textit{divisible on the right} if, for every $B \in \cat{B}$, the functor $\uvar \otimes B \from \cat{A} \to \cat{C}$ admits a right adjoint. When the functor $\uvar \times A \from \C \to \C$ admits a right adjoint, $A$ is an \textit{exponentiable} space. The right adjoint will be denoted $(\uvar)^A \from \C \to \C$, and its action on a morphism $f \from X \to Y$ will be denoted $f_* \from X^A \to Y^A$. A map $g \from A \to B$ between exponentiable spaces induces a natural transformation $g^* \from (\uvar)^B \Rightarrow (\uvar)^A$.

\begin{definition} \label{def:pullback-power}
Let $f \from X \to Y$ and $g \from A \to B$ be maps in $\C$ with $A$ and $B$ both exponentiable. Their \textit{pullback power}, denoted $f \triangleright g$, is the factorization of $f_*g^*=g^*f_*$ in the following diagram:
\[\begin{tikzcd}
	{X^B} \\
	& P & {Y^B} \\
	& {X^A} & {Y^A}
	\arrow["{f_*}", bend left, from=1-1, to=2-3]
	\arrow["{f \triangleright  g}"{description}, dashed, from=1-1, to=2-2]
	\arrow["{g^*}"', bend right, from=1-1, to=3-2]
	\arrow[from=2-2, to=2-3]
	\arrow[from=2-2, to=3-2]
	\arrow["{f_*}"', from=3-2, to=3-3]
	\arrow["{g^*}", from=2-3, to=3-3]
	\arrow["\lrcorner"{anchor=center, pos=0.125}, draw=none, from=2-2, to=3-3]
\end{tikzcd}\]
Note that, unlike the pushout product, in general $f \triangleright g \neq g \triangleright f$.
\end{definition}

The following statement holds in general for \textit{closed} monoidal products in a given category. Since we will not be taking $\C$ to be cartesian closed, we need to ensure that we are only exponentiating spaces that we have assumed or shown to be exponentiable.

\begin{lemma}\label{PushProdAndPullPowerLemma}
Let $i \from A \to B$ be a map between exponentiable spaces. Then $f \hat{\times} i$ has the \llp\ \wrt\ $g$ \iff\ $f$ has the \llp\ \wrt\ $g \triangleright i$
\end{lemma}
\begin{proof}
This follows from \cite{joyal:quasicategories}*{Proposition D.1.18}, since the functor $\uvar \times \uvar \from \C \times \C_{exp} \to \C$ is divisible on the right, where $\C_{exp}$ denotes the full subcategory of $\C$ consisting of exponentiable spaces. 
\end{proof}

\begin{lemma}\label{PushProdIdentity1}
Let $A$ be a space, and $!_A \from \emptyset \to A$ the unique map. For any map $f \from X \to Y$ and a space $A$ we have $!_A \hat{\times} f \iso \id[A] \times f$.
\end{lemma}
\begin{proof}
Immediate from (S4).
\end{proof}

\begin{lemma}\label{PushProdIdentities}
The following identities hold for the pushout-product in the arrow category $\Top^\to$, where $!_X \from \emptyset \to X$ and $i \from \bd I \to I$ is the endpoint inclusions:
\begin{itemize}
\item[(1)] $e_0 \hat{\times} e_0 \iso !_I \hat{\times} e_0$
\item[(2)] $i \hat{\times} e_0 \iso !_I \hat{\times} e_0$
\end{itemize}
\end{lemma}
\begin{proof}
For (1), it follows from \cref{PushProdIdentity1} that we must find an automorphism $\phi \from I\times I \to I \times I$ that restricts to an isomorphism $I \times \{0\} \iso I \times \{0\} \cup \{0\} \times I$, i.e. such that the following commutes:
\[\begin{tikzcd}
	{I \times \{0\}} && {I \times \{0\} \cup \{0\} \times I} \\
	{I \times I} && {I\times I} 
	\arrow[hook, from=1-1, to=2-1]
	\arrow[hook, from=1-3, to=2-3]
	\arrow["{\phi|_{I \times \{0\}}}", from=1-1, to=1-3]
	\arrow["\phi"', from=2-1, to=2-3]
\end{tikzcd}\]
One such $\phi$ is given by the composite $gf$, where $g(x,y)=(\frac{1+x}{2},y)$ and
$$
f(x,y)=\begin{cases}
(x,2y) & \text{if } y \leq x/2\text{;}\\
(2(x-y),x) & \text{if } x/2 \leq y \leq x\text{;}\\
(2(x-y),y) & \text{if } x \leq y \leq 2x\text{;}\\
(-y,2x) & \text{if } y \geq 2x\text{.}
\end{cases}
$$
The proof of (2) is similar.
\end{proof}

Note that, since the pushout-product is well-defined up to isomorphism in the arrow category by \cref{PushProdArrowCategory}, the isomorphism in (1) holds regardless of what pushout we take for $I \coprod_{*} I$, even though we chose a specific pushout for the calculation in (1). In particular, since we have taken $I$ for the pushout $I \coprod_* I$ in $\C$, and $\Box_{\leq 2}$ to be a full subcategory, we may use $(1)$ in any category with intervals.

\section{Q-structures: fibrations and cofibrations} \label{sec:q-structure}

In this and the following section, we will show that the axioms of a Q-structure given in \cref{Q-axioms} are sufficient for $\C$ to admit a model structure defined analogously to the standard model structure on $\Top$.
If $\C$ satisfies the additional requirements for being a good Q-structure, given in (Q7) and (Q8), then the weak equivalences of the model structure are precisely the ones which induce isomorphisms on path-connected components and all homotopy groups, as defined in \cref{HomotopyGroupDef}.

We begin by introducing (relative) cell-complexes.

\begin{definition}
Let $\mathscr{J}$ be a class of maps. A $\mathscr{J}$\textit{-cell complex} is a transfinite composition of pushouts of coproducts of maps in $\mathscr{J}$, as in the following diagram.
\begin{center}
\begin{tikzcd}
	{\coprod \text{dom}(j_\alpha)} & {\coprod \text{cod}(j_\alpha)} \\
	{X_0} & {X_1} & {X_2} & \cdots & {\colim X_k} \\
	& {\coprod \text{dom}(j_\beta)} & {\coprod \text{cod}(j_\beta)}
	\arrow[from=1-1, to=2-1]
	\arrow[from=2-1, to=2-2]
	\arrow["p_1", from=1-2, to=2-2]
	\arrow["\lrcorner"{anchor=center, pos=0.125, rotate=180}, draw=none, from=2-2, to=1-1]
        \arrow["\lrcorner"{anchor=center, pos=0.125, rotate=270}, draw=none, from=2-3, to=3-2]
	\arrow["{\coprod j_\alpha}", from=1-1, to=1-2]
	\arrow["{\coprod j_\beta}"', from=3-2, to=3-3]
	\arrow[from=3-2, to=2-2]
	\arrow["p_2"', from=3-3, to=2-3]
	\arrow[from=2-2, to=2-3]
	\arrow[from=2-3, to=2-4]
	\arrow[from=2-4, to=2-5]
\end{tikzcd}
\end{center}
The collection of all $\mathscr{J}$-cell complexes will be denoted $cell(\mathscr{J})$, and $cof(\mathscr{J})$ will denote the closure of $cell(\mathscr{J})$ under retracts. Let $rlp(\mathscr{J})$ denote the class of maps which have the right lifting property against $\mathscr{J}$. A standard proof shows that $rlp(\mathscr{J}) = rlp(cof(\mathscr{J}))$ \cite{hirschhorn:book}*{Proposition 10.3.2}.
\end{definition}

\begin{notation} \label{notationIandJ}
Let $\mathcal{I} = \{i_n \from \bd I^n \ito I^n\}_{n \geq 0}$ and $\mathcal{J} = \{j_n \from I^n \times \{0\} \ito I^n \times I\}_{n \geq 0}$ be the subspace inclusions in $\Top$. 
\end{notation}

As a slight abuse of notation, let $e_k \from* \ito \bd I$ for $k=0,1$ be the endpoint inclusions.

\begin{definition}\label{Q-axioms}
Let $\iota \from \Box \to \C$ be an embedding of $\Box$ into a bicomplete category $\C$. The pair $(\C, \iota)$ is a \textit{Q-structure} if the following holds:
\begin{enumerate}
\item[(Q1)] $\Box_{\leq 2} \ito \Box \xrightarrow{\iota} \C$ makes $\C$ a category with intervals.
\item[(Q2)] $X \times \bd I$ is a coproduct $X \coprod X$ in $\C$, with inclusions $\id[X] \times e_k$ for all $X \in \C$.
\item[(Q3)] $\mathcal{I} \hat{\times} \mathcal{I} \subseteq cof(\mathcal{I})$ and $\mathcal{I} \hat{\times} \mathcal{J} \subseteq cof(\mathcal{J})$.
\item[(Q4)] $\mathcal{J} \subseteq cof(\mathcal{I})$ and $\{\emptyset \to \bd I^n\} \subseteq cof(\mathcal{I})$.
\item[(Q5)] $\bd I^n$ and $I^n$ are exponentiable in $\C$ for all $n \geq 0$.
\item[(Q6)] $\bd I^{n}$ and $I^{n}$ are small relative to maps in $cell(\mathcal{I})$ and $cell(\mathcal{J})$ respectively, for all $n \geq 0$.
\end{enumerate}
The pair $(\C,\iota)$ is a \textit{good Q-structure} if, in addition,
\begin{enumerate}
\item[(Q7)] $I^n$ is a product of $\prod_{1 \leq k \leq n} I$ for all positive integers $n$.
\item[(Q8)] There are pushouts of the following form in $\C$ for all $n \geq 1$:
\[\begin{tikzcd}
	{\bd I^n} & {\bd I^n \times I} && {\bd I^{n-1}} & {I^{n-1}} \\
	{*} & {I^n} && {*} & {\bd I^n}
	\arrow["c"', from=2-1, to=2-2]
	\arrow["\lrcorner"{anchor=center, pos=0.125, rotate=180}, draw=none, from=2-2, to=1-1]
	\arrow["\rho", from=1-2, to=2-2]
	\arrow["{!}"', from=1-1, to=2-1]
	\arrow["{e_1}", from=1-1, to=1-2]
	\arrow["{!}"', from=1-4, to=2-4]
	\arrow["{i_n}", hook, from=1-4, to=1-5]
	\arrow["{(*)}"', from=2-4, to=2-5]
	\arrow["\sigma", from=1-5, to=2-5]
\end{tikzcd}\]
such that $\rho e_0=i_n$.
\end{enumerate}
\end{definition}

The axioms listed above are key properties used to construct the Quillen model structure on $\Top$, restated for an arbitrary $\C$. We will use the axioms as follows: (Q1) allows us to make use of the homotopy theory developed in \cref{BasicNotions}; (Q2) is used in lifting homotopies along certain maps, with prescribed endpoints; (Q3) and (Q5) allow us to obtain a weak version of the pushout-product axiom; lastly, (Q4) and (Q6) are used in applying the small object argument. For good Q-structures, (Q7) will be used to define homotopy groups (cf. \cref{HomotopyGroupDef}) and (Q8) are some familiar quotient identities from $\Top$, namely $S^{n-1} \times I/S^{n-1} \times \{0\} \iso D^n$ and $D^n/S^{n-1} \iso S^n$.

It follows from (Q5) that $\uvar \times I^n$, as a left adjoint, preserves colimits for all $n \geq 0$. (Q6) is sufficient for $\C$ to admit a small object argument on $\mathcal{I}$ and $\mathcal{J}$, which generates two weak factorization systems in $\C$. For $\mathcal{K}$ equal to either $\mathcal{I}$ or $\mathcal{J}$, one has for the left class $cof(\mathcal{K})$ and for the right class $rlp(\mathcal{K})$ of the respective weak factorization system.

In \cref{Ex:topological_spaces,ex:pseudotop,ex:locales}, we discuss several examples of good Q-structures.
We do not know of any examples of Q-structures that are not good, however we do not expect the axioms (Q1)--(Q6) to imply (Q7)--(Q8), since the latter require the existence of certain pushouts squares which are otherwise not assumed to exist.
We chose to isolate the axioms (Q1)--(Q6) to underscore the fact that they are the only ones required to establish a model structure, while (Q7)--(Q8) are used to characterize the class of weak equivalences.

We can now define the classes of maps that will form a model structure on $\C$.

\begin{definition}\label{CofFibDef} \leavevmode
  \begin{enumerate}
      \item[(1)] A \textit{cofibration} is a map in $cof(\mathcal{I})$.
      \item[(2)] A \textit{fibration} is a map in $rlp(\mathcal{J})$.
      \item[(3)] A \textit{trivial cofibration} is a map in $cof(\mathcal{J})$.
      \item[(4)] A \textit{trivial fibration} is a map in $rlp(\mathcal{I})$.
  \end{enumerate} 
\end{definition}

Using the small object argument, we obtain the expected factorizations.

\begin{lemma} \label{existence-of-factorizations}
    Every map in $\C$ can be factored as:
    \begin{itemize}
        \item[(1)] a cofibration followed by a trivial fibration; and
        \item[(2)] a trivial cofibration followed by a fibration. \qed
    \end{itemize}
\end{lemma}

\begin{definition}\label{WEDef}
Let $f \from X \to Y \in \C$ and $n$ a nonnegative integer. Then
\begin{enumerate}
\item[(1)] $f$ is \emph{n-compressible} if any square of the form
\[\begin{tikzcd}
	{\bd I^n} & X \\
	{I^n} & Y
	\arrow["{i_n}"', hook, from=1-1, to=2-1]
	\arrow["u", from=1-1, to=1-2]
	\arrow["v", from=2-1, to=2-2]
	\arrow["f", from=1-2, to=2-2]
\end{tikzcd}\]
admits a diagonal map $h \from I^n \to X$ such that $hi_n = u$ and $fh \sim v\text{ rel }\bd I^n$.
\item[(2)] $f$ is \emph{n-connected} if it is $k$-compressible for all $k \leq n$. 
\item[(3)] $f$ is a \emph{weak equivalence} if it is $n$-connected for all nonnegative integers.
\end{enumerate}
\end{definition}

A map $f$ has the \textit{weak right lifting property} against $i \from A \to B$ if it has the lifting property as in (1) against $i$ --- given maps such that $fu = vi$ there is a filler $h$ satisfying $hi=u$ and $fh \sim v$ rel $A$.

\begin{theorem}\label{SerreModelStructure}
Any Q-structure $(\C,\iota)$ induces a model structure on $\C$ whose cofibrations and fibrations are as defined in \cref{CofFibDef} and weak equivalences are as defined in \cref{WEDef}.
\end{theorem}

We will defer the proof of this to \cref{SerreSection}.

\begin{definition}
A space $X$ is \textit{fibrant} if the map $X \to *$ is a fibration. Dually, $X$ is \textit{cofibrant} if the map $\emptyset \to X$ is a cofibration.
\end{definition}

\begin{lemma}\label{AllFibrant}
Every space is fibrant.
\end{lemma}
\begin{proof}
Given the following square:
\begin{center}
\begin{tikzcd}
	{I^n} & X \\
	{I^n \times I} & {*}
	\arrow["{j_n}"', from=1-1, to=2-1]
	\arrow["{!}"', from=2-1, to=2-2]
	\arrow["f", from=1-1, to=1-2]
	\arrow["{!}", from=1-2, to=2-2]
\end{tikzcd}
\end{center}
The map $f \pi_{I^n}$ gives a lift.
\end{proof}

\begin{lemma}\label{TrivialFibLemma}
Every trivial fibration is both a fibration and a weak equivalence.
\end{lemma}
\begin{proof}
Clearly if $f$ has the \rlp\ \wrt\ $\mathcal{I}$ it is a weak equivalence, since every square of the form
\begin{center}
\begin{tikzcd}
	{\bd I^n} & X \\
	{I^n} & Y
	\arrow["{i_n}"', hook, from=1-1, to=2-1]
	\arrow["u", from=1-1, to=1-2]
	\arrow["v", from=2-1, to=2-2]
	\arrow["f", from=1-2, to=2-2]
\end{tikzcd}
\end{center}
admits a strict lift, not just one up to homotopy on the lower triangle. Moreover, by (Q4), since $f$ has the \rlp\ \wrt\ all maps in $cof(\mathcal{I})$, it has the right lifting property \wrt\ $\mathcal{J}$, thus is a fibration.
\end{proof}

\begin{proposition}\label{PushProdLemma}
The following identities hold in $\C$:
\begin{enumerate}
\item[(1)] $cof(\mathcal{I}) \hat{\times} \mathcal{I} \subseteq cof(\mathcal{I})$
\item[(2)] $cof(\mathcal{I}) \hat{\times} \mathcal{J} \subseteq cof(\mathcal{J})$
\item[(3)] $\mathcal{I} \hat{\times} cof(\mathcal{J}) \subseteq cof(\mathcal{J})$
\end{enumerate}
\end{proposition}
\begin{proof}
Since $\mathcal{I} \hat{\times} \mathcal{I} \subseteq cof(\mathcal{I})$ by (Q3), $\mathcal{I} \hat{\times} \mathcal{I}$ has the \llp\ \wrt\ $rlp(\mathcal{I})$, so $\mathcal{I}$ has the \llp\ \wrt\ $rlp(\mathcal{I}) \triangleright \mathcal{I}$ by \cref{PushProdAndPullPowerLemma}. Thus, $cof(\mathcal{I})$ has the \llp\ \wrt\ $rlp(\mathcal{I}) \triangleright \mathcal{I}$ hence $cof(\mathcal{I}) \hat{\times} \mathcal{I} \subseteq cof(\mathcal{I})$ by \cref{PushProdAndPullPowerLemma} again. The other parts are analogous.
\end{proof}

For closed monoidal products $\otimes$, we can obtain a stronger result that $cof(\mathcal{I}) \hat{\otimes} cof(\mathcal{J}) \subseteq cof(\mathcal{K})$ whenever $\mathcal{I} \hat{\otimes} \mathcal{J} \subseteq \mathcal{K}$. In this case, maps in $cof(\mathcal{I})$ might not be between exponentiable spaces, so we cannot apply the same reasoning again. Thus, some care will be needed when applying \cref{PushProdLemma} to ensure that at least one map is in either $\mathcal{I}$ or $\mathcal{J}$.

\begin{definition}\label{OpenBoxInclusion}
The \textit{open box inclusion} is the induced map in the following pushout on the right
\[\begin{tikzcd}
	{*} & I && {*} & I \\
	I & I && I & I \\
	&& {I^2} &&& {I^2}
	\arrow["{e_0}", from=1-1, to=1-2]
	\arrow["{e_1}"', from=1-1, to=2-1]
	\arrow[from=2-1, to=2-2]
	\arrow[from=1-2, to=2-2]
	\arrow["\lrcorner"{anchor=center, pos=0.125, rotate=180}, draw=none, from=2-2, to=1-1]
	\arrow["{[\bdL,\bdT]}"{description}, dashed, from=2-2, to=3-3]
	\arrow["\bdT", bend left, from=1-2, to=3-3]
	\arrow["\bdL"', bend right, from=2-1, to=3-3]
	\arrow["{e_1}"', from=1-4, to=2-4]
	\arrow[from=2-4, to=2-5]
	\arrow["{e_0}", from=1-4, to=1-5]
	\arrow[from=1-5, to=2-5]
	\arrow["\lrcorner"{anchor=center, pos=0.125, rotate=180}, draw=none, from=2-5, to=1-4]
	\arrow["{[[\bdL,\bdT],-\bdR]}"{description, pos=0.3}, dashed, from=2-5, to=3-6]
	\arrow["{-\bdR}", bend left, from=1-5, to=3-6]
	\arrow["{[\bdL,\bdT]}"', bend right, from=2-4, to=3-6]
\end{tikzcd}\]
where $[-, -]$ denotes the induced map from the pushout.
\end{definition}

By \cref{PushProdArrowCategory} and \cref{PushProdIdentities}, we recognize the left pushout as $[\bdL,\bdT] \iso e_0 \hat{\times} e_0 \iso !_I \hat{\times} e_0 \iso \id[I] \times e_0 = j_1$, so by the same reasoning, $[[\bdL,\bdT],-\bdR] \iso j_1$ in the arrow category. When we wish to make the distinction, we will use $\sqcap$ for $I$ in the above right pushout and $\inc \from \sqcap \ito I^2$ for $[[\bdL,\bdT],-\bdR]$. Geometrically, it should be thought of as the left, top, and right edges of the square, which we will use to ``fill out" homotopies. By \cref{PushProdLemma}, given a cofibration $i \from A \cto X$ we have $i \hat{\times} \inc$ is a trivial cofibration. To verify that two maps from $\sqcap$ in this context are equal, one of which is a composition including $\inc$, it suffices to check that they agree after precomposing with $b$, $aa$, and $ab$ (cf.~\cref{S-axioms}).

\begin{lemma}\label{GradLemma}
Let $f \from X \to Y$ be a homotopy equivalence. Then there is a map $g \from Y \to X$, along with homotopies $H$ from $\id[X]$ to $gf$ and $K$ from $\id[Y]$ to $fg$ such that there is a map $\alpha \from X \times I^2 \to Y$ such that $\alpha \bdL=fH$, $\alpha \bdT=\const_{fgf}$, $\alpha \bdR=Kf$, and $\alpha\bdB=\const_f$.
\end{lemma}
\begin{proof}
This is the statement that any equivalence can be promoted to an adjoint equivalence, applied to the 2-category $\C$ as given in \cref{Bicategory}.
\end{proof}

\begin{theorem}\label{HEAreWE}
Homotopy equivalences are weak equivalences.
\end{theorem}
\begin{proof}
Let $f \from X \to Y$ be a homotopy equivalence and $g$, $H$, $K$, and $\alpha$ be as in the statement of \cref{GradLemma}. Suppose that for a nonnegative integer $n$ the following diagram commutes
\begin{center}
\begin{tikzcd}
	{\bd I^n} & X \\
	{I^n} & Y
	\arrow["i_n"', hook, from=1-1, to=2-1]
	\arrow["u", from=1-1, to=1-2]
	\arrow["f", from=1-2, to=2-2]
	\arrow["v"', from=2-1, to=2-2]
\end{tikzcd}
\end{center}
By \cref{PushProdLemma}, $i_n \hat{\times} e_1 \in cof(\mathcal{J})$, so there is a lift $J \from I^n \times I \to X$ since $X$ is fibrant:
\[\begin{tikzcd}
	{\bd I^n \times I \coprod_{\bd I^n \times \{1\}} I^{n} \times \{1\}} && X \\
	{I^n \times I}
	\arrow["i_n \hat{\times} e_1"', hook, from=1-1, to=2-1]
	\arrow["{[Hu, gv]}", from=1-1, to=1-3]
	\arrow["J"', dashed, from=2-1, to=1-3]
\end{tikzcd}\]
Let $w = Je_0 \from I^n \to X$. Since $He_0 = \id[X]$, we have $wi_n = Je_0i_n = J(i_n \times \id[I])e_0 = Hue_0 = u$. By the lifting triangle above, $fJi_n = fHu$, and by the initial square, $Kvi_n = Kfu$. Taking the convention for $\inc \from \sqcap \ito I^2$ as in \cref{OpenBoxInclusion}, the following identities hold:
\begin{align*} 
\alpha u \circ \inc \circ b &= \alpha u \circ -\bdR = -Kfu = -Kvi_n\text{;}\\
\alpha u \circ \inc \circ aa &= \alpha u \bdL = fHu = fJi_n\text{;}\\
\alpha u \circ \inc \circ ab &= \alpha u \bdT = \const_{fgfu}=\const_{fgv}i_n \text{,}
\end{align*}
hence the top map in the diagram below is defined, so it admits a lift $\tilde{\alpha} \from I^n \times I^2 \to Y$ by \cref{PushProdLemma}
\[\begin{tikzcd}
	{\bd I^n \times I^2 \coprod_{\bd I^n \times \sqcap} I^n \times \sqcap} &&&&& Y \\
	{I^n \times I^2}
	\arrow["{i_n \hat{\times} \inc}"', from=1-1, to=2-1]
	\arrow["{\tilde{\alpha}}"', dashed, from=2-1, to=1-6]
	\arrow["{[\alpha u,\ [[fJ,\ \const_{fgv}],\ -Kv]]}", from=1-1, to=1-6]
\end{tikzcd}\]
The composite $\tilde{\alpha} \bdB$ gives a homotopy from $fw$ to $v$ rel $\bd I^n$, since $\tilde{\alpha} \bdB i_n = \tilde{\alpha}(i_n \times \id[I^2]) \bdB = \alpha u\bdB= \const_{fu}= \const_{vi_n}$.
\end{proof}

\section{Quillen model structure}\label{SerreSection}

In this section, we put all the pieces together and prove \cref{SerreModelStructure}, while also characterizing weak equivalences in terms of homotopy groups (\cref{WEClassicDef}), which are introduced in \cref{HomotopyGroupDef}.

\begin{lemma}
For a weak equivalence $f \from X \to Y$ and a cofibration $i \from A \cto B$, if the following square commutes:
\[\begin{tikzcd}
	A & X \\
	B & Y
	\arrow["i"', tail, from=1-1, to=2-1]
	\arrow["f", from=1-2, to=2-2]
	\arrow["u", from=1-1, to=1-2]
	\arrow["v"', from=2-1, to=2-2]
\end{tikzcd}\]
then there is a filler $h \from B \to X$ such that $hi=u$ and $fh \sim v$ rel $A$.
\end{lemma}
\begin{proof}
This is analogous to the proof that the left class of a weak factorization system is closed under coproducts, pushouts, transfinite composition, and retracts, except special care is needed to check that relativity is preserved in each part. To get the induced homotopy on the colimits, one needs to use relativity and cocontinuity of $\uvar \times I$. We will only give the proof for pushouts. Suppose that the square on the left below is a pushout and that $f$ has the weak \rlp\ (cf. \cref{WEDef}) against $i$
\[\begin{tikzcd}
	A & {A'} && {A'} & X \\
	B & {B'} && {B'} & Y
	\arrow["i"', from=1-1, to=2-1]
	\arrow["\alpha", from=1-1, to=1-2]
	\arrow["\beta"', from=2-1, to=2-2]
	\arrow["{i'}", from=1-2, to=2-2]
	\arrow["{i'}"', from=1-4, to=2-4]
	\arrow["u", from=1-4, to=1-5]
	\arrow["v"', from=2-4, to=2-5]
	\arrow["f", from=1-5, to=2-5]
	\arrow["\lrcorner"{anchor=center, pos=0.125, rotate=180}, draw=none, from=2-2, to=1-1]
\end{tikzcd}\]
Given the square on the right above, we can combine the two squares to get a diagonal filler $h \from B \to X$ with $hi=u\alpha$ and a homotopy $H \from B \times I \to Y$ from $fh$ to $v\beta$ rel $A$. This induces a map $h' \from B' \to X$ in the diagram on the left and a homotopy $H' \from B' \times I \to Y$ on the right by cocontinuity of $\uvar \times I$
\[\begin{tikzcd}
	A & {A'} && {A \times I} && {A'\times I} \\
	B & {B'} && {B \times I} && {B' \times I} \\
	&& X &&&& Y
	\arrow["i"', from=1-1, to=2-1]
	\arrow["{i'}", from=1-2, to=2-2]
	\arrow["\alpha", from=1-1, to=1-2]
	\arrow["\beta"', from=2-1, to=2-2]
	\arrow["{h'}"{description}, dashed, from=2-2, to=3-3]
	\arrow["h"', from=2-1, to=3-3, bend right]
	\arrow["u", from=1-2, to=3-3, bend left]
	\arrow["{\beta \times \id[I]}"', from=2-4, to=2-6]
	\arrow["{i' \times \id[I]}"', from=1-6, to=2-6]
	\arrow["{\alpha \times \id[I]}", from=1-4, to=1-6]
	\arrow["{i \times \id[I]}"', from=1-4, to=2-4]
	\arrow["{H'}"{description}, dashed, from=2-6, to=3-7]
	\arrow["H"', from=2-4, to=3-7, bend right]
	\arrow["{\const_{fu}}", from=1-6, to=3-7, bend left]
\end{tikzcd}\]
since $Hi=\const_{fhi}=\const_{fu\alpha}=\const_{fu}\alpha$ by relativity of $H$. By commutativity, $H'i'=\const_{fu}$ so $H'$ is relative to $A'$, and one can verify through the pushout on the right that $H'e_0=fh'$ and $H'e_1=v$.
\end{proof}

\begin{definition}
Let $\ho\C_{cof}$ denote the full subcategory of $\ho\C$ (\cref{ClSphDef}) consisting of cofibrant spaces.
\end{definition}

\begin{theorem}\label{WECofibrantObj}
A map $f \from X \to Y$ is a weak equivalence \iff
\[f_* \from [A,X] \to [A,Y]\tag{$*$}\]
is a bijection for all cofibrant spaces $A$.
\end{theorem}
\begin{proof}
Suppose first that $f$ is a weak equivalence, and let $A$ be a cofibrant space. By the above lemma, for any $[\beta] \in [A,Y]$, there is a filler $\alpha \from A \to X$ in the diagram
\[\begin{tikzcd}
	\varnothing & X \\
	A & Y
	\arrow["f", from=1-2, to=2-2]
	\arrow["{!}", from=1-1, to=1-2]
	\arrow["{!}"', from=1-1, to=2-1]
	\arrow["\beta"', from=2-1, to=2-2]
\end{tikzcd}\]
such that $f \alpha \sim \beta$, hence $f_*$ is surjective. If $f_*[\alpha] = f_*[\beta]$ for $[\alpha], [\beta] \in [A,X]$, then there is a homotopy $H \from A \times I \to Y$ from $f \alpha$ to $f \beta$. Since $A$ is cofibrant, $!_A \hat{\times} i_1=\id[A] \times i_1 \from A \times \partial I \to A \times I$ is a cofibration by \cref{PushProdLemma} and \cref{PushProdIdentity1}. By (Q2), $A \times \partial I$ is a coproduct $A \coprod A$, so the square
\[\begin{tikzcd}
	{A \times \bd I} & X \\
	{A \times I} & Y
	\arrow["{!_A \hat{\times} i_1}"', from=1-1, to=2-1]
	\arrow["{\alpha + \beta}", from=1-1, to=1-2]
	\arrow["H"', from=2-1, to=2-2]
	\arrow["f", from=1-2, to=2-2]
\end{tikzcd}\]
admits a filler $K \from A \times I \to X$ such that the upper triangle commutes, so $K$ is a homotopy from $\alpha$ to $\beta$, hence $f_*$ is injective. 

Suppose instead that $f_* \from [A,X] \to [A,Y]$ is a bijection for all cofibrant spaces $A$.
Using \cref{existence-of-factorizations}, factor $\emptyset \to X$ as a cofibration followed by a trivial fibration, which is a weak equivalence and a fibration by \cref{TrivialFibLemma}: $\emptyset \cto \tilde{X} \afto X$. Similarly, factor $\emptyset \to Y$ into $\emptyset \cto \tilde{Y} \afto Y$. Then there is a lift $\tilde{f} \from \tilde{X} \to \tilde{Y}$ in the following diagram
\[\begin{tikzcd}
	\emptyset & {\tilde{Y}} \\
	{\tilde{X}} & Y
	\arrow[tail, from=1-1, to=2-1]
	\arrow["{^{\small\sim}}"{marking}, two heads, from=1-2, to=2-2]
	\arrow[from=2-1, to=2-2]
	\arrow[tail, from=1-1, to=1-2]
        \arrow["{\tilde{f}}"{description}, dashed, from=2-1, to=1-2]
\end{tikzcd}\]
where the bottom arrow is the composite $\tilde{X} \afto X \xrightarrow{f} Y$. Since we have already shown that weak equivalences satisfy $(*)$, it follows that $\tilde{f}$ does as well, by the 2-out-of-3 property for isomorphisms, since the following diagram commutes
\[\begin{tikzcd}
	{\tilde{X}} & X \\
	{\tilde{Y}} & Y
	\arrow["{\tilde{f}}"', from=1-1, to=2-1]
	\arrow["{_{_{\small\sim}}}", two heads, from=2-1, to=2-2]
	\arrow["{_{_{\small\sim}}}", two heads, from=1-1, to=1-2]
	\arrow["f", from=1-2, to=2-2]
\end{tikzcd}\]
Then $[\tilde{f}]_* \from \ho\C_{cof}(\uvar,\tilde{X}) \Rightarrow \ho\C_{cof}(\uvar, \tilde{Y})$ is a natural isomorphism, so it follows from the Yoneda Lemma that $[\tilde{f}]$ is an isomorphism in $\ho\C_{cof}$, hence a homotopy equivalence and a weak equivalence by \cref{HEAreWE}. Suppose now that the following diagram commutes:
 \[\begin{tikzcd}
	{\bd I^n} & X \\
	{I^n} & Y
	\arrow["{i_n}"', from=1-1, to=2-1]
	\arrow["u", from=1-1, to=1-2]
	\arrow["v"', from=2-1, to=2-2]
	\arrow["f", from=1-2, to=2-2]
\end{tikzcd}\]
Since $\bd I^n$ is cofibrant by (Q4), we may choose lifts for the following squares:
\[\begin{tikzcd}
	\emptyset & {\tilde{X}} && {\bd I^n} & {\tilde{Y}} \\
	{\bd I^n} & X && {I^n} & Y
	\arrow["{^{\small\sim}}"{marking}, two heads, from=1-2, to=2-2]
	\arrow["u"', from=2-1, to=2-2]
	\arrow[tail, from=1-1, to=2-1]
	\arrow[from=1-1, to=1-2]
	\arrow["{i_n}"', tail, from=1-4, to=2-4]
	\arrow["{^{\small\sim}}"{marking}, two heads, from=1-5, to=2-5]
	\arrow["{\tilde{f}\tilde{u}}", from=1-4, to=1-5]
	\arrow["v"', from=2-4, to=2-5]
	\arrow["{\tilde{u}}"{description}, dashed, from=2-1, to=1-2]
	\arrow["{\tilde{v}}"{description}, dashed, from=2-4, to=1-5]
\end{tikzcd}\]
Then we obtain a map $\tilde{h} \from I^n \to \tilde{X}$ such that $\tilde{h} i_n = \tilde{u}$ and $\tilde{f}\circ\tilde{h} \sim \tilde{v}$ rel $\bd I^n$ in the following diagram
\[\begin{tikzcd}
	{\bd I^n} & {\tilde{X}} & X \\
	{I^n} & {\tilde{Y}} & Y
	\arrow["{i_n}"', from=1-1, to=2-1]
	\arrow["{\tilde{u}}", from=1-1, to=1-2]
	\arrow["{\tilde{v}}"', from=2-1, to=2-2]
	\arrow["{\tilde{f}}", from=1-2, to=2-2]
	\arrow[two heads, from=2-2, to=2-3]
	\arrow[two heads, from=1-2, to=1-3]
	\arrow["f", from=1-3, to=2-3]
\end{tikzcd}\]
Composing $\tilde{h}$ with $\tilde{X} \afto X$ gives the required map.
\end{proof}

\begin{proposition}\label{WE2OutOf3}
Weak equivalences satisfy the 2-out-of-3 property.
\end{proposition}
\begin{proof}
This is immediate from \cref{WECofibrantObj} and the 2-out-of-3 property for isomorphisms.
\end{proof}

\begin{theorem}\label{SOA1}
A map $f \from X \to Y$ is a trivial fibration \iff\ $f$ is a weak equivalence and a fibration.
\end{theorem}
\begin{proof}
The forward direction is given in \cref{TrivialFibLemma}. Suppose instead that $f$ is a weak equivalence and a fibration. Given a square of the form
\begin{center}
\begin{tikzcd}
	{\bd I^n} & X \\
	{I^n} & Y
	\arrow["{i_n}"', hook, from=1-1, to=2-1]
	\arrow["u", from=1-1, to=1-2]
	\arrow["v", from=2-1, to=2-2]
	\arrow["f", from=1-2, to=2-2]
\end{tikzcd}
\end{center}
Choose diagonal filler $w \from I^n \to X$ such that $wi_n=u$ and a homotopy $H \from I^n \times I \to Y$ from $fw$ to $v$ relative to $\bd I^n$. Since $i_n \hat{\times} j_0 \in cof(\mathcal{J})$ by \cref{PushProdLemma}, we may choose a lift $\tilde{H} \from I^n \times I \to X$ for the following square
\begin{center}
\begin{tikzcd}
	{\bd I^n \times I \coprod_{\bd I^n \times \{0\}} I^n \times \{0\}} && X \\
	{I^n \times I} && Y
	\arrow["{i_n \hat{\times} j_0}"', hook, from=1-1, to=2-1]
	\arrow["{[\const_u, w]}", from=1-1, to=1-3]
	\arrow["f", from=1-3, to=2-3]
	\arrow["H"', from=2-1, to=2-3]
	\arrow["{\tilde{H}}"{description}, dashed, from=2-1, to=1-3]
\end{tikzcd}
\end{center}
Commutativity of this square comes from $He_0=fw$ and $H$ being relative to $\bd I^n$. Let $\tilde{w}=\tilde{H}e_1$. Then $\tilde{w}i_n=\tilde{H}i_ne_1=\const_ue_1=u$ and $f\tilde{w}=He_1=v$, so $\tilde{w}$ is the required lift.
\end{proof}

\begin{theorem}\label{SOA2}
A map $f \from X \to Y$ is a trivial cofibration \iff\ $f$ is a weak equivalence and a cofibration.
\end{theorem}
\begin{proof}
It follows from (Q4) that $cof(\mathcal{J}) \subseteq cof(\mathcal{I})$, so every trivial cofibration is a cofibration. Suppose $f$ is a trivial cofibration. Since $X$ is fibrant, we obtain a lift $p \from Y \to X$ in the following diagram
\begin{center}
\begin{tikzcd}
	X & X \\
	Y & {*}
	\arrow["f"', from=1-1, to=2-1]
	\arrow["{!}", from=1-2, to=2-2]
	\arrow["{!}"', from=2-1, to=2-2]
	\arrow["{\text{id}_X}", from=1-1, to=1-2]
        \arrow["{p}"{description}, dashed, from=2-1, to=1-2]
\end{tikzcd}
\end{center} 
such that $pf = \id[X]$. We show that $fp \sim \id[Y]$, from which it will follow from \cref{HEAreWE} that $f$ is a weak equivalence. Since $f \hat{\times} i_1 \in cof(\mathcal{J})$ by \cref{PushProdLemma} and $Y \times \bd I$ is a coproduct $Y \coprod Y$ by (Q2), we obtain a lift $H \from Y \times I \to Y$ in the following
\begin{center}
\begin{tikzcd}
	{X \times I\coprod_{X \times \bd I} (Y \times \bd I)} &&& Y \\
	{Y \times I} &&& {*}
	\arrow["{!}", from=1-4, to=2-4]
	\arrow["{!}"', from=2-1, to=2-4]
	\arrow["f \hat{\times} i_1"', from=1-1, to=2-1, hook]
	\arrow["{[\const_f,\ \text{id}_Y + fp]}", from=1-1, to=1-4]
	\arrow["{H}"{description}, dashed, from=2-1, to=1-4]
\end{tikzcd}
\end{center}
By commutativity, $H$ is a homotopy from $fp$ to $\id[Y]$.

The converse follows immediately by the retract argument.
\end{proof}

We are now ready to show that the three classes given in \cref{CofFibDef} and \cref{WEDef} define a model structure on $\C$.

\begin{proof}[Proof of \cref{SerreModelStructure}]
By (Q6), $\C$ admits two weak factorization systems. The first, cofibrantly generated by $\mathcal{I}$ has cofibrations for its left class by definition and, by \cref{SOA1}, the intersection of fibrations and weak equivalences for its right class. The second, cofibrantly generated by $\mathcal{J}$ has, by \cref{SOA2}, the intersection of cofibrations and weak equivalences for its left class and fibrations for its right class by definition. Moreover, by \cref{WE2OutOf3}, weak equivalences satisfy the 2-out-of-3 property.
\end{proof}

\subsection*{Homotopy groups in a good Q-structure}

For the remainder of this section, we assume that $(\C, \iota)$ is a good Q-structure. For a space $X \in \C$, a \textit{point} $x_0$ in $X$ is a map $x_0 \from * \to X$; if there is a map $* \to X$, then $X$ \textit{admits points}. We will denote the composite $X \to * \xrightarrow{y_0} Y$ by $\const_{y_0}$ or just $y_0$. Let $\C_*$ be the slice category under $*$: its objects are pairs $(X, x_0 \from * \to X)$ and a morphism $f \from (X,x_0) \to (Y,y_0)$ is a based map $f \from X \to Y$ such that $f(x_0)=y_0$.
Given two maps $f,g \from (X,x_0) \to (Y,y_0)$, a homotopy between them in this category is a homotopy $H$ from $f$ to $g$ relative to $x_0$.
As in \cref{BasicNotions}, this defines an equivalence relation on $\C_*((X,x_0),(Y,y_0))$.

Recall that, in (Q7), there is a point $(*) \from * \to \bd I^n$ in the second pushout, which we will treat as a distinguished base point of $\bd I^n$ to define homotopy groups. Such a point is given for all $n \geq 1$, though we will not make the distinction between $(*)$ for different $n$.

\begin{definition}\label{HomotopyGroupDef} 
Let $n \geq 0$ and $x_0 \from * \to X$ a point in $X$. Let $\pi_n(X, x_0) = \C_*((\bd I^{n+1}, (*)), (X,x_0))/\sim$ denote the set of based maps $(\bd I^{n+1}, (*)) \to (X,x_0)$ quotiented by homotopy rel $(*)$. By the pushout in (Q8), this is equivalently the set of all maps $f \from I^n \to X$ such that $fi_n=\const_{x_0}$, quotiented by homotopy relative to $\partial I^n$. For $n \geq 1$, this set admits a group operation, with which it is the $n^{th}$ \textit{homotopy group} of $X$ at $x_0$. The operation on $\pi_n(X,x_0)$ is defined by $[\alpha] \bk [\beta] = [\alpha \bk \beta]$, where $\alpha \bk \beta$ is the induced map in
\[\begin{tikzcd}
	{I^{n-1}} && {I^n} \\
	{I^n} && {I^n} \\
	&&& X
	\arrow["{e_0}", from=1-1, to=1-3]
	\arrow["{e_1}"', from=1-1, to=2-1]
	\arrow["{\id[I^{n-1}] \times u}"', from=2-1, to=2-3]
	\arrow["{\id[I^{n-1}] \times v}"', from=1-3, to=2-3]
	\arrow["{\alpha \bk \beta}"{description}, dashed, from=2-3, to=3-4]
	\arrow["\alpha"', bend right, from=2-1, to=3-4]
	\arrow["\beta", bend left, from=1-3, to=3-4]
\end{tikzcd}\]
since $\alpha e_1 = \beta e_0 = \const_{x_0}$.
\end{definition}

Note that we are implicitly using (Q7) and (S3) in defining the group operation above to recognize $I^n \iso I^{n-1} \times I$. We will continue to make this identification without further reference. As in the topological case, $\pi_0(X,x_0)$ will not form a group, but rather acts as the set of path-connected components.
\begin{theorem}
The operation on $\pi_n(X,x_0)$ defined above makes it a group for $n \geq 1$. 
\end{theorem}
\begin{proof}
The identity is $\const_{x_0}$ by \cref{ConstantHomotopyIdentity}, associativity comes from \cref{HomotopyCompAssociative}, and the inverse of $[\alpha]$ is $[-\alpha]$ by \cref{HomotopyInverseComposition}.
\end{proof}

\begin{lemma}\label{RelativeHomotopyLemma}
Let $X$ be a space and $i \from A \cto X$ a cofibration. Suppose $f,g,h \from X \to Y$ are continuous maps such that $fi=gi$. Given homotopies $H$ from $f$ to $h$ and $K$ from $g$ to $h$, if $Hi \sim Ki$ rel endpoints then there is a homotopy from $f$ to $g$ that is relative to $A$.
\end{lemma}
\begin{proof}
Choose a map $\alpha \from A \times I^2 \to Y$ such that $\alpha\bdL=Hi$, $\alpha\bdT=\const_{hi}$, $\alpha\bdR=Ki$, and $\alpha \bdB=\const_{fi}$. By \cref{PushProdLemma}, $i \hat{\times} \inc$ is in $cof(\mathcal{J})$, where $\inc \from \sqcap \to I^2$ is as in \cref{OpenBoxInclusion}. Since the following identities hold
\begin{align*} 
\alpha \circ \inc \circ b &= \alpha \circ -\bdR = -Ki\text{;}\\
\alpha \circ \inc \circ aa &= \alpha \bdL = Hi\text{;}\\
\alpha \circ \inc \circ ab &= \alpha \bdT = \const_{hi}\text{,}
\end{align*}
the top map is defined, so there is a lift $\beta \from X \times I^2 \to Y$ in the following diagram
\[\begin{tikzcd}
	{A \times I^2 \coprod_{A \times \sqcap} X \times \sqcap} &&&&& Y \\
	{X \times I^2}
	\arrow["i \hat{\times} \inc"', from=1-1, to=2-1]
	\arrow["\beta"', dashed, from=2-1, to=1-6]
	\arrow["{[\alpha,\ [[Hi,\ \const_{hi}],\ -Ki]]}", from=1-1, to=1-6]
\end{tikzcd}\]
The composite $\beta\bdB$ gives a homotopy from $f$ to $g$ that is relative to $A$, since $\beta\bdB i=\beta (i \times \id[I^2]) \bdB=\alpha \bdB = \const_{fi}$.
\end{proof}

\begin{lemma}\label{HomotopyGroupLemma}
Suppose $(\C, \iota)$ is a good Q-structure, and let $n \geq 1$. Then a map $u \from \bd I^n \to X$ is homotopic to $x$ rel $(*)$ \iff\ there is a map $\tilde{u} \from I^n \to X$ such that $\tilde{u}i_n=u$.
\end{lemma}
\begin{proof}
Suppose $u$ is homotopic to the point $x=u(*)$ rel $(*)$. Then there is a homotopy $H \from \bd I^n \times I \to X$ such that $He_0 = u$ and $He_1=\const_x$. Thus, by the left pushout in (Q8), $H$ factors through $I^n$ to a map $\tilde{u} \from I^n \to X$ such that $\tilde{u}i_n=He_0=u$. Conversely, suppose that $u$ admits an extension $\tilde{u} \from I^n \to X$. Let $H \from I^n \times I \to I^n$ be $H(s,t)=(1-t)s+t\cdot(*)$, the straight line homotopy from $\id[I^n]$ to $(*)$, which satisfies $H \circ (*) = \const_{(*)}$. Then $\tilde{H} = \tilde{u}Hi_n$ satisfies $\tilde{H}e_0=u$, $\tilde{H}e_1=u(*)$, and $\tilde{H}(*)=\const_{u(*)}=\const_x$, so $\tilde{H}$ is the required homotopy relative to $(*)$.
\end{proof}

\begin{lemma}
Suppose $(\C, \iota)$ is a good Q-structure and that $X$ admits points. Let $f \from X \to Y$ be continuous and $n$ a positive integer. Then
\begin{itemize}
\item[(1)] $f$ is $0$-compressible \iff\ $\pi_0 f$ is surjective.
\item[(2)] $f$ is $n$-compressible \iff\ for all points $x$ in $X$, $\pi_k f \from \pi_k(X,x) \to \pi_k(Y,f(x))$ is injective for $k=n-1$ and surjective for $k=n$.
\end{itemize}
\end{lemma}
\begin{proof}
For (1), since $\bd I^{0}=\emptyset$ and $I^0\iso*$ by (S1), the diagram in \cref{WEDef} always admits a diagonal filler \iff\ for any point $y$ in $Y$, we may choose a point $x$ in $X$ such that $f x$ is homotopic to $y$; that is, $\pi_0f$ is surjective. For (2), suppose first that $f$ is $n$-compressible. Let $x \in X$ and fix $[\beta] \in \pi_n(Y,f(x))$. Regarding $\beta$ as a map $I^n \to Y$ such that $\beta i_n = f(x)$, we get a diagonal filler $\alpha \from I^n \to X$ for the following square
\begin{center}
\begin{tikzcd}
	{\bd I^n} & X \\
	{I^n} & Y
	\arrow["i_n"', hook, from=1-1, to=2-1]
	\arrow["x", from=1-1, to=1-2]
	\arrow["f", from=1-2, to=2-2]
	\arrow["\beta"', from=2-1, to=2-2]
\end{tikzcd}
\end{center}
where $\alpha i_n=x$, hence $[\alpha] \in \pi_n(X,x)$, and $f\alpha \sim \beta$ rel $\bd I^n$, hence $\pi_nf([\alpha])=[\beta]$. Thus, $\pi_n f$ is surjective. For injectivity, if $n-1=0$, then if $\pi_0f[x] = [f(x)]=[f(x)']= \pi_0f[x']$ in $\pi_0Y$ for $[x],[x'] \in \pi_0X$, then choosing a homotopy $p$ from $f(x)$ to $f(x)'$ induces a diagonal filler in 
\[\begin{tikzcd}
	\bd I & X \\
	I & Y
	\arrow["f", from=1-2, to=2-2]
	\arrow["p"', from=2-1, to=2-2]
	\arrow["{[x,x']}", from=1-1, to=1-2]
	\arrow["i_1"', hook, from=1-1, to=2-1]
\end{tikzcd}\]
since $\bd I$ is a coproduct $* \coprod *$ by (Q2), which gives a homotopy from $x$ to $x'$. Otherwise, if $[f\alpha] = [\const_{f(x)}]$ for $[\alpha] \in \pi_{n-1}(X,x)$, choose an extension $\tilde{f\alpha} \from I^{n-1} \to Y$ given by \cref{HomotopyGroupLemma}, so there is a diagonal filler in 
\[\begin{tikzcd}
	{\bd I^{n-1}} & X \\
	{I^{n-1}} & Y
	\arrow["f", from=1-2, to=2-2]
	\arrow["i_{n-1}"', hook, from=1-1, to=2-1]
	\arrow["\alpha", from=1-1, to=1-2]
	\arrow["{\tilde{f\alpha}}"', from=2-1, to=2-2]
\end{tikzcd}\]
which by \cref{HomotopyGroupLemma} we have $[\alpha] = 0$, hence $\pi_kf$ is injective. 
Conversely, suppose that $\pi_k$ is injective for $k=n-1$ and surjective for $k=n$. Given the following commutative square
\begin{center}
\begin{tikzcd}
	{\bd I^n} & X \\
	{I^n} & Y
	\arrow["{i_n}", hook, from=1-1, to=2-1]
	\arrow["u", from=1-1, to=1-2]
	\arrow["f", from=1-2, to=2-2]
	\arrow["v"', from=2-1, to=2-2]
\end{tikzcd}
\end{center}
Let $x = u(*)$. By commutativity and \cref{HomotopyGroupLemma}, $fu$ is homotopic to $\const_{f(x)}$ rel $(*)$, which by injectivity of $\pi_{n-1}f$ gives that $u$ is homotopic to $\const_x$, say through a homotopy $H \from \bd I^n \times I \to X$ rel $(*)$. Since $fu = vi_n$ and $Y$ is fibrant, by \cref{PushProdLemma} there is a lift in the diagram
\[\begin{tikzcd}
	{\bd I^n \times I \cup I^{n} \times \{0\}} && Y \\
	{I^n \times I}
	\arrow["i_n \hat{\times} j_0"', hook, from=1-1, to=2-1]
	\arrow["{[fH,v]}", from=1-1, to=1-3]
	\arrow["J"', dashed, from=2-1, to=1-3]
\end{tikzcd}\]
where the left map is in $cof(\mathcal{J})$ by \cref{PushProdLemma}. Let $\tilde{v} = Je_1$. Since $\tilde{v}i_n=Ji_ne_1=fHe_1=\const_{f(x)}$, we can regard $\tilde{v}$ as a map on $\bd I^{n+1}$ by the right pushout in (Q8), so $[\tilde{v}] \in \pi_n(Y,f(x))$. Since $\pi_n f$ is surjective, there is a $[\tilde{w}] \in \pi_n(X,x)$ such that there is a homotopy $K \from I^n \times I \to Y$ from $f\tilde{w}$ to $\tilde{v}$ rel $\bd I^n$. Regarding $\tilde{w}$ as a map on $I^n$, we obtain a lift in the diagram
\[\begin{tikzcd}
	{\bd I^n \times I \cup I^{n} \times \{1\}} && Y \\
	{I^n \times I}
	\arrow[hook, from=1-1, to=2-1]
	\arrow["{[H,\tilde{w}]}", from=1-1, to=1-3]
	\arrow["G"', dashed, from=2-1, to=1-3]
\end{tikzcd}\]
since $He_1 = \const_{x}$. Let $w = Ge_0$, so $wi_n = Gi_ne_0=He_0=u$, and $vi_n=fu=fwi_n$. Since $J\bk -K $ is a homotopy from $v$ to $f\tilde{w}$, $fG$ is a homotopy from $fw$ to $f\tilde{w}$, and
$$
(J\bk -K)i_n=Ji_n\bk-Ki_n=fH\bk \const_{f(x)} \sim fH = fGi_n \text{ rel endpoints}
$$
we may apply \cref{RelativeHomotopyLemma} to get that $fw \sim v$ rel $\bd I^n$, hence $f$ is $n$-connected.
\end{proof}

Thus, we may conclude that, in a good Q-structure $(\C, \iota)$, the definition of weak equivalences given in \cref{WEDef} is equivalent to the classical definition given in terms of homotopy groups.

\begin{theorem}\label{WEClassicDef} 
In a good Q-structure, if a space $X$ admits points then a map $f \from X \to Y$ is n-connected \iff\ for all points $x$ in $X$, $\pi_kf \from \pi_k(X,x) \to \pi_k(Y,f(x))$ is an isomorphism for $k < n$ and a surjection for $k=n$. In particular, if $X$ admits points then a map $f \from X \to Y$ is a weak equivalence \iff\ $\pi_n f \from \pi_n(X,x) \to \pi_n(Y,f(x))$ is an isomorphism for every nonnegative integer $n$ and every point $x$ in $X$. \qed
\end{theorem}

\section{Example: subcategories of topological spaces}\label{Ex:topological_spaces}

We begin by recovering the usual Quillen model structure on the category $\Top$ of topological spaces.
From this, we will also be able to verify that the result holds for several subcategories of $\Top$, as well as for pseudotopological spaces in $\PsTop$ and locales in $\Loc$.

\begin{theorem} \label{TopModelStructure}
    The category $\Top$ of topological spaces carries a model structure whose weak equivalences are the weak homotopy equivalences, fibrations are the Serre fibrations, and cofibrations are the Serre cofibrations.
\end{theorem}

To prove the theorem, we will appeal to \cref{SerreModelStructure} with the obvious inclusion $\iota \from \Box \ito \Top$. We begin however with a technical lemma which verifies the only tricky axiom (Q3).

\begin{lemma}\label{PushProdInTop}
For $n$ and $m$ nonnegative integers, $i_n \hat{\times} i_m \iso i_{n+m}$ and $i_n \hat{\times} j_m \iso j_{n+m}$ in the arrow category $\Top^\to$. Thus, in \cref{notationIandJ}, $\mathcal{I} \hat{\times} \mathcal{I} \subseteq \mathcal{I}$ and $\mathcal{I} \hat{\times} \mathcal{J} \subseteq \mathcal{J}$.
\end{lemma}

\begin{proof}
In $\Top$, we may take the pushout-product of two inclusions $A \ito B$ and $C \ito D$ to be the inclusion $A \times D \cup B \times C \ito B \times D$. Thus, $i_n \hat{\times} i_m$ is the inclusion $\bd I^n \times I^m \cup I^n \times \bd I^m \ito I^{n+m}$. Since $\bd I^n \times I^m$ is the subspace of $I^{n+m}$ where at least one of the first $n$ coordinates are either $0$ or $1$, and likewise with $I^n \times \bd I^m$ with the last $m$ coordinates, we have $\bd I^n \times I^m \cup I^n \times \bd I^m \iso \bd I^{n+m}$. Thus, $i_n \hat{\times} i_m \iso i_{n+m}$. Similarly, $i_n \hat{\times} j_m$ is the inclusion $\bd I^n \times I^m \times I \cup I^n \times I^m \times \{0\} \ito I^{n+m} \times I$. But $\bd I^n \times I \cup I^n \times \{0\} \iso I^n$, so $i_n \hat{\times} j_0 \iso j_0$ and thus since $\uvar \times I^m$ preserves colimits, we have $i_n \hat{\times} j_m \iso (I^n \times I^m \ito I^{n+m} \times I) \iso j_{n+m}$.
\end{proof}

\begin{proof}[Proof of \cref{TopModelStructure}]
    Having fixed the inclusion $\iota \from \Box \ito \Top$, (S1), (S2), and (S4) hold, and (S3) is given by letting $a(t)=t/2$ and $b(t)=(1+t)/2$. Moreover, (Q2) and (Q7) are obvious, (Q5) holds since $\bd I^n$ and $I^n$ are locally compact Hausdorff, and (Q6) is well known (see for instance \cite{hovey:book}*{Proposition 2.4.2}). For (Q3), we then use \cref{PushProdInTop}.
    
    (Q4) follows from the given pushouts
\[\begin{tikzcd}
	{\bd I^n} & {I^n} &&&& {\bd I^{n-1}} & {I^{n-1}} \\
	{I^n} & {\bd I^{n+1}} & {I^{n+1}} && \emptyset & {*} & {\bd I^n}
	\arrow["{i_{n+1}}"', hook, from=2-2, to=2-3]
	\arrow["{i_n}", hook, from=1-1, to=1-2]
	\arrow[from=1-2, to=2-2]
	\arrow["\lrcorner"{anchor=center, pos=0.125, rotate=180}, draw=none, from=2-2, to=1-1]
	\arrow["{i_n}"', hook, from=1-1, to=2-1]
	\arrow["{\in cell(\mathcal{I})}"', from=2-1, to=2-2]
	\arrow["{i_0}"', hook, from=2-5, to=2-6]
	\arrow[from=1-6, to=2-6]
	\arrow["{i_{n-1}}", hook, from=1-6, to=1-7]
	\arrow[from=1-7, to=2-7]
	\arrow["\lrcorner"{anchor=center, pos=0.125, rotate=180}, draw=none, from=2-7, to=1-6]
	\arrow["{\in cell(\mathcal{I})}"', from=2-6, to=2-7]
\end{tikzcd}\]
Lastly, (Q8) is precisely the statements that under the identification $\bd I^n \times I/\bd I^n \times \{1\} \iso I^n$ and $I^{n-1} / \bd I^{n-1} \iso \bd I^n$, if we take $\rho$ and $\sigma$ to be the quotient maps, then any continuous map that is constant on fibers of the quotient factors through the quotient. By the description of the classes of maps, it is clear this is the classic Quillen model structure.
\end{proof}

In a (co)reflective subcategory of $\Top$, much of the work has already been done. In particular, if $\C$ is a (co)reflective subcategory of $\Top$ containing the CW complexes, then we may take $\iota \from \Box \ito \C$ to be the canonical inclusion. In this case, the image of $\iota$ is naturally isomorphic to the image of $\Box$ under the (co)reflection, hence we may take each of the above pushouts to be the same as in $\Top$. All that remains is to check that $X \times \uvar$ preserves the pushout in (S3), as well as the conditions of (Q5) and (Q6). We introduce 4 (co)reflective subcategories which admit good Q-structures.

Let $X$ be a topological space. Recall that a closed subset $F \subseteq X$ is \textit{irreducible} if it cannot be written as the union of two proper closed subsets of $F$. That is, if $F = F_1 \cup F_2$ and $F_1$ and $F_2$ are closed, then one of $F_1 = F$ or $F_2 = F$.

\begin{definition}\label{TopologicalSubcategories}
We take the following subcategories to all be full:
\begin{enumerate}
\item[(1)] Let $\ncat{cgHaus}$ be the subcategory of compactly generated Hausdorff spaces.
\item[(2)] Let $\ncat{CGWH}$ be the subcategory of compactly generated weakly Hausdorff spaces.
\item[(3)] Let $\DTop$ be the subcategory of $\Delta$-generated spaces. A $\Delta$-generated space $X$ is a topological space which has the final topology \wrt\ all maps from simplices $\simp{n} \to X$. There is an obvious functor $\Delta\operatorname{-ify} \from \Top \to \DTop$, which simply refines the topology on a space $X$ to be the finest topology for which all maps $\simp{n} \to X$ are continuous. This forms an adjunction $\iota \adj \Delta\operatorname{-ify}$, making $\DTop$ a coreflective subcategory of $\Top$ \cite{vogt:convenient}*{Corollary 1.4}.
\item[(4)] Let $\ncat{Sob}$ be the subcategory of sober spaces. A sober space $X$ is a topological space such that every irreducible closed subset of $X$ is the closure of exactly one point. In particular, every Hausdorff space is sober. There is a soberification functor $\operatorname{sober} \from \Top \to \ncat{Sob}$, which is a left adjoint to the inclusion $\iota \from \ncat{Sob} \ito \Top$, making $\ncat{Sob}$ a reflective subcategory of $\Top$.
\end{enumerate}
\end{definition} 

There is a chain of (co)reflections relating $\ncat{cgHaus}$ and $\ncat{CGWH}$ to $\Top$, so the comments above apply to them.
The categories $\ncat{cgHaus}$, $\ncat{CGWH}$, and $\DTop$ are cartesian closed (see for instance \cite{steenrod:convenient}, \cite{cgwh:spaces}, and \cite{vogt:convenient} respectively). In $\ncat{Sob}$, the exponentiable spaces are precisely the locally compact spaces \cite{stone:spaces}*{Theorem VII.4.12}. Thus, in each of these categories, (Q5) is satisfied. A standard proof (such as the one in \cite{hovey:book}) shows that (Q6) is also satisfied in $\ncat{cgHaus}$ and $\ncat{CGWH}$, and $\DTop$ is locally presentable \cite{directed:homotopy}*{Theorem 3.7} hence admits the small object argument. 

Most of the model structures on the above categories are well-known (see for instance the proof in \cite{hirschhorn:model-structure-on-top}). The only model structure that we are unaware of an existing reference for is on $\Sob$, so we will carefully check the remaining axiom (Q6).

\begin{lemma}\label{CellComplexLemma} 
Let $\mathcal{K}$ be either $\mathcal{I}$ or $\mathcal{J}$. Then every map in $cell(\mathcal{K})$ is a closed $T_1$ inclusion; that is, if $f \from X \to Y \in cell(\mathcal{K})$, then $f$ is a closed map such that every point in $Y-fX$ is closed. Moreover, $Y-fX$ is Hausdorff in the subspace topology.
\end{lemma}
\begin{proof}
The first part is given in \cite{hovey:book}*{Lemma 2.4.5}. For the second part, notice that $Y/fX$ is a cell complex and hence Hausdorff, so the result follows since $Y-fX \iso Y/fX - [fX]$ as subspaces.
\end{proof}

\begin{lemma}
Coproducts of sober spaces, as formed in $\Top$ are sober.
\end{lemma}
\begin{proof}
Let $\coprod X_\alpha$ be a coproduct in $\Top$, where each $X_\alpha$ is sober. Clearly, any irreducible closed subset must be contained within exactly one component of $\coprod X_\alpha$, say $X_\beta$, hence it is the closure of a single point since $X_\beta$ is sober.
\end{proof}

\begin{lemma}\label{SoberSpaceLemma}
Suppose $X$ is a space and $A, B \subseteq X$ are subsets such that $X$ is the disjoint union of $A$ and $B$ as sets, $A$ is a closed in $X$, and every point of $B$ is closed in $X$. If $A$ is sober and $B$ is Hausdorff under the subspace topologies, then $X$ is sober.
\end{lemma}
\begin{proof}
Let $F \subseteq X$ be an irreducible closed subset of $X$. Suppose $F$ has at least two distinct points of $B$, say $x$ and $y$. Then since $B$ is Hausdorff, there exists disjoint sets $U \subseteq B$ and $V \subseteq B$ that are open neighborhoods of $x$ and $y$ respectively in $B$. Since $A$ is closed, $B$ is open, so $U$ and $V$ are open in $X$. Then $F = (F-U) \cup (F-V)$ gives $F$ as a union of two nonempty closed subsets of $X$, contradicting irreducibility. Thus, $F$ has at most one point of $B$. If $F$ intersected both $A$ and $B$, then it would contain exactly one point of $B$, say $x$, and since $\{x\}$ is closed in $X$, we would have $F = (F \cap A) \cup \{x\}$, contradicting irreduciblity. Thus, $F$ is either a single point of $B$ or an irreducible closed subset of $A$; in either case, $F$ is the closure of exactly one point since $A$ is sober.
\end{proof} 

\begin{corollary}
Let $\mathcal{K}$ be either $\mathcal{I}$ or $\mathcal{J}$, and let $\{k_\alpha \from S_\alpha \to K_\alpha\}$ be a collection of maps in $\mathcal{K}$. If $X$ is sober and the following is a pushout in $\Top$, then $Y$ is sober.
\[\begin{tikzcd}
	{\coprod_\alpha S_\alpha} & {\coprod_\alpha K_\alpha} \\
	X & Y
	\arrow["f"', hook, from=2-1, to=2-2]
	\arrow["{\coprod k_\alpha}", hook, from=1-1, to=1-2]
	\arrow["\lrcorner"{anchor=center, pos=0.125, rotate=180}, draw=none, from=2-2, to=1-1]
	\arrow[from=1-2, to=2-2]
	\arrow[from=1-1, to=2-1]
\end{tikzcd}\]
\end{corollary}
\begin{proof}
This follows from \cref{CellComplexLemma} and \cref{SoberSpaceLemma}, given that $Y$ is the disjoint union of $fX$ and $(Y-fX)$ as sets.
\end{proof}

In particular, this means that anytime cells are attached to a sober space in $\ncat{Sob}$ in a single step, we may take the resulting relative cell complex, as formed in $\Top$ to be the pushout. In fact, this then holds for transfinite composites of such maps as well:

\begin{lemma}
Let $\mathcal{K}$ be either $\mathcal{I}$ or $\mathcal{J}$. Let $X_0 \xrightarrow{f_0} X_1 \xrightarrow{f_1} \cdots \to X = \colim X_\alpha$ be an transfinite composite of pushouts of coproducts of maps in $\mathcal{K}$, formed in $\Top$. If $X_0$ is sober, then so is $X$.
\end{lemma}
\begin{proof}
Again, this holds by \cref{CellComplexLemma} and \cref{SoberSpaceLemma}.
\end{proof}

\begin{corollary}
In $\ncat{Sob}$, compact Hausdorff spaces are small relative to maps in $cell(\mathcal{I})$ and $cell(\mathcal{J})$.
\end{corollary}
\begin{proof}
By the above lemmas, transfinite composites of pushouts of maps in $\mathcal{I}$ and $\mathcal{J}$ are formed the same in $\ncat{Sob}$ as in $\Top$. Thus, the maps in $cell(\mathcal{I}_\ncat{Sob})$ and $cell(\mathcal{J}_\ncat{Sob})$ are precisely those in $cell(\mathcal{I}_\Top)$ and $cell(\mathcal{J}_\Top)$ whose domain is sober. The result follows from the usual statement in $\Top$.
\end{proof}

It is also easily verified that (S3), (Q7), and (Q8) hold in each of the categories in \cref{TopologicalSubcategories}, making each a good Q-structure with the canonical inclusions. Thus, we may apply \cref{SerreModelStructure} and \cref{WEClassicDef} for each of the categories to get the following:

\begin{theorem}
Each of the categories of:
\begin{itemize}
    \item compactly generated Hausdorff spaces $\ncat{cgHaus}$,
    \item compactly generated weakly Hausdorff spaces $\ncat{CGWH}$,
    \item $\Delta$-generated spaces $\DTop$
    \item sober spaces $\ncat{Sob}$
\end{itemize}
admits a model structure by taking the weak homotopy equivalences, Serre fibrations, and relative cell complexes to be the weak equivalences, fibrations, and cofibrations. \qed
\end{theorem}

Many of the model structures above are known to be Quillen equivalent to the standard model structure on $\Top$. The only remaining one to our knowledge is the model structure on $\ncat{Sob}$, which we easily verify is Quillen equivalent as well.

\begin{lemma}\label{WESober}
A map $f \from A \to B$ in $\ncat{Sob}$ is a weak equivalence \iff\ $\iota f$ is in $\Top$.
\end{lemma}
\begin{proof}
We have characterized weak equivalences in both categories by maps from sober spaces. The result follows from $\ncat{Sob}$ being full.
\end{proof}

\begin{theorem}
The adjunction $\operatorname{sober} \adj \iota$ is a Quillen equivalence.
\end{theorem}
\begin{proof}
That it is a Quillen adjunction holds by construction, since $\operatorname{sober}(\mathcal{I}_\Top) \iso \mathcal{I}_\ncat{Sob}$, and likewise with $\mathcal{J}$, hence $\op{sober}$ preserves (trivial) cofibrations since it preserves colimits. Thus, let $X \in \Top$ be cofibrant, which is sober as a Hausdorff space, and $Y \in \ncat{Sob}$ be any sober space. By the following diagram and \cref{WESober}, $\operatorname{sober}X \to Y$ is a weak equivalence \iff\ $X \to  \iota Y$ is:
\[\begin{tikzcd}
	{\iota \circ \operatorname{sober}X} & \iota Y \\
	X
	\arrow[from=1-1, to=1-2]
	\arrow["\iso", from=2-1, to=1-1] 
	\arrow[from=2-1, to=1-2]
\end{tikzcd}\]
\end{proof}

\section{Example: pseudotopological spaces} \label{ex:pseudotop}

Pseudotopological spaces are a generalization of topological spaces, and a more specialized type of convergence spaces. Rather than describing the structure in terms of open sets, one takes convergence of ultrafilters to be foundational. The category of pseudotopological spaces $\PsTop$ has several nice properties which $\Top$ lacks, most notably the fact that $\PsTop$ is quasitopos \cite{MR0911688}. In fact, \cite{improving:top}*{Theorem 3.12} showed that $\PsTop$ is the cartesian closed topological hull of the category of pretopological spaces $\ncat{PreTop}$ (also known as \v{C}ech closure spaces) in the category of convergence spaces $\ncat{Conv}$. In particular, since $\ncat{DiGraph}$ and $\ncat{Graph}$ both fully embed into $\ncat{PreTop}$ \cite{bubenik:homotopy}*{Proposition 1.56}, the existance of a model structure on $\PsTop$ could be applicable to the homotopy theory of graphs and \v{C}ech closure spaces. We now properly introduce pseudotopological spaces and give a brief recollection of the necessary lemmas.

\begin{definition}
Let $X$ be a set, $\mathbb{F}X$ be the set of filters on $X$, and $\beta X$ be the set of ultrafilters on $X$. For each $x \in X$, let $x_\bullet = \set{A \subseteq X}{x \in A}$ be the principal ultrafilter at $x$. A \textit{pseudotopology} on $X$ is a function $\lim_\xi \from \beta X \to \mathcal{P}X$ such that $x \in \lim_\xi x_\bullet$ for all $x \in X$. A \textit{peudotopological space} is a pair $(X, \lim_\xi)$ such that $\lim_\xi$ is a peudotopology on $X$. 
\end{definition}

Given a pseudotopology, we can define the limit of any filter $\mathscr{F}$ on $X$:
$$
\lim_\xi \mathscr{F} = \bigcap_{\mathscr{U} \in \beta \mathscr{F}} \lim_\xi \mathscr{U}\text{,}
$$
where $\beta \mathscr{F} = \set{\mathscr{U} \in \beta X}{\mathscr{F} \subseteq \mathscr{U}}$. By the ultrafilter lemma, $\beta \mathscr{F}$ is nonempty unless $\mathscr{F} = \mathcal{P}X$, in which case the empty intersection is understood to be $X$ (that is, $\mathcal{P}X$ converges to every element of $X$).

\begin{definition}
Given pseudotopologies $\lim_\xi$ on $X$ and $\lim_\zeta$ on $Y$, a function $f \from X \to Y$ is \textit{continuous} if $f(\lim_\xi \mathscr{F}) \subseteq \lim_\zeta f_*\mathscr{F}$ for every $\mathscr{F} \in \mathbb{F} X$, where $f_*\mathscr{F}=\set{B \subseteq Y}{\text{there is } A \in \mathscr{F} \text{ such that } fA \subseteq B}$ is the pushforward of $\mathscr{F}$. The category of pseudotopological spaces and continuous functions will be denoted $\PsTop$.
\end{definition}

\begin{definition}
Given two pseudotopologies $\lim_\xi$ and $\lim_\zeta$ on $X$, $\lim_\xi$ is \textit{finer} than $\lim_\zeta$ (respectively $\lim_\zeta$ is coarser than $\lim_\xi$, written $\zeta \leq \xi$) if, for all $\mathscr{U} \in \beta X$, $\lim_\xi \mathscr{U} \subseteq \lim_\zeta \mathscr{U}$. Equivalently, $\lim_\xi$ is finer than $\lim_\zeta$ if and only if $\id[X] \from (X,\lim_\xi) \to (X,\lim_\zeta)$ is continuous.
\end{definition}

\begin{lemma}\cite{convergence:foundations}*{Section III.6}
For any set $X$, the collection of pseudotopologies $\Xi$ on $X$ is a complete lattice.
\end{lemma}
\begin{proof}
Let $\mathscr{U}$ be an ultrafilter on $X$, and $\Gamma \subseteq \Xi$. Then an infimum is given by $\lim_{\bigwedge\Gamma} \mathscr{U}=\bigcup_{\xi\in\Gamma} \lim_\xi \mathscr{U}$ and a supremum is given by $\lim_{\bigvee\Gamma}\mathscr{U}=\bigcap_{\xi\in\Gamma} \lim_\xi \mathscr{U}$.
\end{proof}

The existence of a finest and a coarsest pseudotopology implies the existence of discrete and indiscrete functors from $\Set$ to $\PsTop$. Indeed, these form the usual discrete $\adj$ forgetful $\adj$ indiscrete adjunctions, so any (co)limit in $\PsTop$ has, as its underlying set, the (co)limit set. Specifically, for the discrete pseudotopology only the pricipal ultrafilter $x_\bullet$ converges to $x$, and for the indiscrete pseudotopology, every ultrafilter converges to every point. Moreover, for any set of maps $\{f_k \from X_k \to Y\}_{k\in K}$, we may take the \textit{final pseudotopology} on $Y$, which is the finest pseudotopology for which each $f_k$ is continuous \cite{convergence:foundations}*{{\S}IV.3}.
Dually, we may take the \textit{initial pseudotopology} with respect to a set of maps with common domain, which is the coarsest pseudotopology for which each map is continuous.

\begin{corollary}
$\PsTop$ is (co)complete.
\end{corollary}
\begin{proof}
Take the (co)limit of the underlying sets, and then take the initial (final) pseudotopology under the (co)limit cone.
\end{proof}





For any pair of pseudotopological spaces $X$ and $Y$, there is a natural pseudotopology on $\PsTop(X,Y)$. We do not describe it here instead referring the reader to \cite{improving:top}*{Theorem 3.7}.

\begin{proposition}\cite{improving:top}*{Theorem 3.7} \label{PsTopCartesianClosed}
The category $\PsTop$ is cartesian closed. \qed
\end{proposition}

Every topology $\tau$ on $X$ induces a pseudotopology, by taking $\lim_\tau \mathscr{U} = \set{x \in X}{\mathcal{N}(x) \subseteq \mathscr{U}}$; that is, one takes the usual limits of the filter. Given a topology $\sigma$ on $Y$, a function $f \from X \to Y$ is continuous with respect to the topologies \iff\ it is continuous with respect to the induced pseudotopologies. Thus, we get a full embedding $\iota \from \Top \ito \PsTop$. In the other direction, given any pseudotopology, we can take its topological modification.

\begin{definition}
Let $(X,\lim_\xi)$ be a pseudotopological space. A set $O \subseteq X$ is \emph{open} if $\lim_\xi \mathscr{U} \cap O \neq \emptyset$ implies $O \in \mathscr{U}$.
A set is \emph{closed} if its complement is open.
\end{definition}

One can verify that the collection of open sets $\tau_\xi$ forms a topology which is the \textit{topological modification} of $\lim_\xi$; the topological modification of a space $(X,\lim_\xi)$ is $(X, \tau\xi)$. One can verify that this process preserves continuity, so this defines a functor $\tau \from \PsTop \to \Top$, which is left adjoint to $\iota$. This makes $\Top$ a reflective subcategory of $\PsTop$. In particular, since the underlying sets and set functions have not changed, the unit is the identity map $\id[X] \from (X,\lim_\xi) \to (X, \tau\lim_{\xi})$. Through this, we are able to verify most of the required axioms from the required identities holding in $\Top$ through a few lemmas.

\begin{definition}\cite{convergence:foundations}*{Definition III.3.1, Proposition IV.2.14}
Let $(X, \lim_\xi)$ be a pseudotopological space, and $i \from A \ito X$ a set inclusion. The \textit{subspace pseudotopology}, denoted $\lim_{\xi|_A}$, is the initial pseudotopology under $i$. Explicitly, this is given by $x \in \lim_{\xi|_A} \mathscr{F}$ \iff\ $x \in \lim_\xi i_*\mathscr{F}$.
\end{definition}

\begin{lemma}\cite{convergence:foundations}*{Proposition IV.2.15}\label{SubspaceRestiction}
Let $f \from (X, \lim_\xi) \to (Y, \lim_\zeta)$ be continuous, and $A \subset X$, $B \subset Y$ such that $fA \subseteq B$. Then $f|_A \from (A, \lim_{\xi|_A}) \to (B, \lim_{\zeta|_B})$ is continuous. \qed
\end{lemma} 

\begin{lemma}\cite{convergence:foundations}*{Proposition V.4.24}\label{SubspaceLemma}
For any psuedotopological space $(X, \lim_\xi)$ and any set-inclusion $i \from A \ito X$, the topological modification of the subspace pseudotopology, $\tau(\xi|_A)$, is finer than the subspace topology under the topological modification, $(\tau\xi)|_A$. That is, $\id[A] \from (A, \lim_{\tau(\xi|_A)}) \to (A, \lim_{(\tau\xi)|_A})$ is continuous. \qed
\end{lemma}

\begin{definition}
Two collections of subsets $\mathcal{A},\mathcal{B} \subseteq \mathcal{P}X$ \textit{mesh}, written $\mathcal{A} \# \mathcal{B}$, if $A \cap B \neq \emptyset$ for every $A \in \mathcal{A}$ and $B\in \mathcal{B}$. The \textit{grill} of $\mathcal{A}$ is 
$$
\mathcal{A}^\# = \set{B \subseteq X}{A \cap B \neq \emptyset \text{ for all }A \in \mathcal{A}}\text{.}
$$
\end{definition}

\begin{definition}
Let $(X,\lim_\xi)$ be a pseudotopological space. For $\mathcal{A} \subseteq \mathcal{P}X$, the \textit{adherence} of $\mathcal{A}$ is 
\[
\adh_\xi\mathcal{A}=\bigcup_{\mathscr{F} \# \mathcal{A}} \lim_\xi\mathscr{F}
\]
where the union is taken over all filters $\mathscr{F}$ that mesh with $\mathcal{A}$. Given two subsets $A,B \subseteq X$, $A$ is $\xi$-\textit{compact at B} if, for every filter $\mathscr{F}$, $A \in \mathscr{F}^\#$ implies $\adh_\xi \mathscr{F} \cap B \neq \emptyset$. $A$ is \textit{compact} if $A$ is $\xi$-compact at $A$, and \textit{relatively compact} if $A$ is $\xi$-compact at $X$. 
\end{definition}

Note that if $A$ is $\xi$-compact at $B$ and $B \subseteq C$, then $A$ is $\xi$-compact at $C$. Moreover, if $A$ and $B$ are $\xi$-compact at $C$, then $A \cup B$ is compact at $C$. A topological space is compact in the above sense \iff\ it is compact in the usual sense.

\begin{lemma}\cite{convergence:foundations}*{Proposition IX.1.26}\label{CompactImageCompact} 
Let $f \from (X,\lim_\xi) \to (Y,\lim_\zeta)$ be continuous. For $A,B \subseteq X$, if $A$ is $\xi$-compact at $B$, then $fA$ is $\zeta$-compact at $fB$. \qed
\end{lemma}

\begin{proposition}\cite{convergence:foundations}*{Corollary IX.1.18}\label{CompactSpacesBalanced} 
Let $X$ be a compact pseudotopological space and $Y$ a Hausdorff topological space. If $f \from X \to Y$ is a continuous bijection, then $f$ is a homeomorphism. \qed
\end{proposition}

\begin{corollary}\label{PushoutsInPsTop}
Suppose the following is a pushout diagram in $\Top$:
\[\begin{tikzcd}
	A & B \\
	C & P
	\arrow["j"', from=2-1, to=2-2]
	\arrow["i", from=1-2, to=2-2]
	\arrow["f", from=1-1, to=1-2]
	\arrow["g"', from=1-1, to=2-1]
	\arrow["\lrcorner"{anchor=center, pos=0.125, rotate=180}, draw=none, from=2-2, to=1-1]
\end{tikzcd}\]
where $B$ and $C$ are compact and $P$ is Hausdorff. Then $P$ is the pushout in $\PsTop$.
\end{corollary}
\begin{proof}
Let $\lim_\xi$ be the colimit pseudotopology on $P$, and $\tau$ the colimit topology. By \cref{CompactImageCompact}, $iB$ and $jC$ are $\xi$-compact in $P$, so $P=iB \cup jC$ is compact in $P$. The result follows from \cref{CompactSpacesBalanced}, since by the adjunction $\tau \adj \iota$, $\id[P] \from (P,\lim_\xi) \to (P, \lim_\tau)$ is continuous.
\end{proof}

We are now ready to show that $\PsTop$ carries a model structure.
Note that this is identical to the model structure defined in \cite{rieser:pseudotopology}, but constructed in a different manner.

\begin{theorem}\label{PsTopMS}
There is a model structure on $\PsTop$ whose cofibrations are maps in $cof(\mathcal{I})$, fibrations are maps in $rlp(\mathcal{J})$, and weak equivalences are maps which have the weak right lifting property against $\mathcal{I}$, or equivalently by \cref{WEClassicDef}, maps which induce isomorphisms on all homotopy groups.
\end{theorem}

\begin{proof}
We are now ready to verify the axioms of \cref{Q-axioms}. That (S1) and (S4) hold are obvious, (S2) follows from $\iota$ being a right adjoint, and (S3) by \cref{PushoutsInPsTop} and \cref{PsTopCartesianClosed}. Thus, $(\C,\iota)$ is a category with intervals. For (Q2), note that $* \coprod *$ is given the finest pseudotopology for which the inclusions $* \ito * \coprod *$ are continuous. Since these are always continuous, $* \coprod *$ must have the discrete topology, hence $* \coprod * \iso \bd I$. The general case follows from \cref{PsTopCartesianClosed}. Next, (Q3) and (Q4) follow from \cref{PushProdInTop} and \cref{PushoutsInPsTop}, and (Q5) by \cref{PsTopCartesianClosed}. The verification of (Q6) is done in \cite{rieser:pseudotopology}*{Theorem 5.16}. Lastly, (Q7) follows from $\iota$ preserving limits, and (Q8) by \cref{PushoutsInPsTop}.
\end{proof}

Of course, since only limits of topological spaces are again topological in $\PsTop$, not every cofibration in $\Top$ is necessarily a cofibration in $\PsTop$. In fact, this needs not hold for cofibrant cell complexes even. 

Next, we will verify that this model structure is Quillen equivalent to the standard model structure on $\Top$, via the adjunction $\tau \adj \iota$. We will first check that it is a Quillen adjunction as well as some intermediate lemmas.

\begin{proposition}
The adjunction $\tau \adj \iota$ is a Quillen adjunction.
\end{proposition}
\begin{proof}
Let $\mathcal{K}$ be either $\mathcal{I}$ or $\mathcal{J}$, and $\mathcal{K}_{\Top}$ and $\mathcal{K}_{\PsTop}$ be the respective generating (acyclic) cofibrations. Clearly, $\tau \mathcal{K}_{\PsTop}=\mathcal{K}_\Top$. Since $\tau$ is a left adjoint, it preserves colimits, so $\tau (cell(\mathcal{K}_\PsTop)) \subseteq cell(\mathcal{K}_\Top)$. Since functors preserve retracts in the arrow category, it follows that $\tau$ preserves (trivial) cofibrations.
\end{proof}

\begin{lemma}\label{TopInclusionPsTopWELemma}
A map $f \from X \to Y$ between topological spaces is a weak equivalence in $\PsTop$ \iff\ it is a weak equivalence in $\Top$.
\end{lemma}
\begin{proof}
Since we have characterized weak equivalences by maps from topological spaces via lifting properties (and for homotopies, $\iota$ preserves products), this follows from $\Top$ being a full subcategory of $\PsTop$.
\end{proof}

Let $\emptyset \cto A$ be a cofibrant cell complex in $\PsTop$, and fix a presentation of $A$ as a transfinite composite, $A = \colim A_\alpha$. For each cell attached to $A$, there is a \textit{characteristic map}, denoted $\Phi_k \from I^{n_k} \to A$, which is the composite $I^{n_k} \ito \coprod_k I^{n_k} \to A_\alpha \to A$. A subset $S \subseteq A$ is a (finite) subcomplex if it can be written as a (finite) union, $S = \bigcup_k \Phi_k I^{n_k}$.

\begin{lemma}\label{PsTopFiniteSubcomplexLemma}
Let $(X, \lim_\xi)$ be a cofibrant cell complex in $\PsTop$, and $S$ a finite subcomplex of $X$. Then the subspace pseudotopology on $S$ agrees with the subspace topology on $\tau X$. 
\end{lemma}
\begin{proof}
Let $S$ be given by the union $S = \bigcup_{k=1}^N \Phi_k I^{n_k}$, and let $\lim_{\xi|_S}$ be the subspace pseudotopology under $\lim_\xi$ and $\lim_{(\tau\xi)|_S}$ be the subspace topology under $\tau\xi$. By adjunction, $\id[S] \from (S, \lim_{\xi|_S}) \to (S, \lim_{\tau(\xi|_S)})$ is continuous, and by \cref{SubspaceLemma}, $\id[S] \from (S, \lim_{\tau(\xi|_S)}) \to (S, \lim_{(\tau\xi)|_S})$ is as well; thus their composite is continuous. 
Since $\tau X$ is a cell complex in $\Top$, it is Hausdorff, and since the restricted maps $\Phi_k \from I^{n_k} \to (S,\lim_{\xi|_S})$ are continuous by \cref{SubspaceRestiction}, their image is compact in $S$. The result follows from \cref{CompactSpacesBalanced}: since $S$ is a finite union of sets compact in $\lim_{\xi|_S}$, $S$ itself is compact in $\lim_{\xi|_S}$, and $\lim_{(\tau\xi)|_S}$ is Hausdorff as a subspace of $\tau X$.
\end{proof}

\begin{theorem}
The adjunction $\tau \adj \iota$ is a Quillen equivalence between the standard model structure on $\Top$ and the model structure on $\PsTop$ given in \cref{PsTopMS}.
\end{theorem}
\begin{proof}
By \cite{hovey:book}*{Corollary 1.3.16}, since $\iota$ preserves weak equivalences by \cref{TopInclusionPsTopWELemma}, it suffices to check that $\id[X] \from X \to \iota\tau X=\tau X$ is a weak equivalence for cofibrant $X$. The general case will follow from considering just $\mathcal{I}$-cell complexes. Suppose $X$ is a cofibrant $\mathcal{I}$-cell complex and that the following diagram commutes
\[\begin{tikzcd}
	{\bd I^n} & X \\
	{I^n} & {\tau X}
	\arrow["{\id[X]}", from=1-2, to=2-2]
	\arrow["{i_n}"', from=1-1, to=2-1]
	\arrow["v"', from=2-1, to=2-2]
	\arrow["u", from=1-1, to=1-2]
\end{tikzcd}\]
Since $\tau X$ is a cell complex in $\Top$ and $v I^n$ is a compact subset, it is contained in a finite subcomplex, say $S$. By \cref{PsTopFiniteSubcomplexLemma}, the subspace topology under $\tau X$ and subspace pseudotopology under $X$ agree, so the inclusion maps $S \ito X$ and $S \ito \tau X$ are both continuous. Factoring $v$ through $S$, we get the following diagram, which gives a strict lift:
\[\begin{tikzcd}
	{\bd I^n} && X \\
	& S \\
	{I^n} && {\tau X}
	\arrow["{i_n}"', hook, from=1-1, to=3-1]
	\arrow["{\id[X]}", from=1-3, to=3-3]
	\arrow["v"{description}, from=3-1, to=2-2]
	\arrow["v"', from=3-1, to=3-3]
	\arrow[hook, from=2-2, to=3-3]
	\arrow[hook, from=2-2, to=1-3]
	\arrow["u", from=1-1, to=1-3]
\end{tikzcd}\]
\end{proof}

Of course, the approach in this paper is not suitable for every topological category. For instance, in the category $\ncat{PreTop}$ of pretopological spaces \cite{bubenik:homotopy,rieser:cech} (also referred to as \v{C}ech closure spaces), which is reflective in $\ncat{PsTop}$ and contains $\Top$ as a reflective subcategory, most of the axioms assumed are suitable, except for the exponentiability requirement. In particular, \cite{exponential:prtop}*{Theorem 3.4} showed that a pretopological space $X$ is exponentiable \iff\ every point in $X$ has a smallest neighborhood, which of course means $I$ is not exponentiable. Similarly, the only exponentiable $T_0$ uniform spaces are the discrete spaces \cite{niefield:cartesian}*{Corollary 3.4}, hence \cref{SerreModelStructure} is not applicable to the category of uniform spaces either.  

\section{Example: locales} \label{ex:locales}

As the main application of \cref{SerreModelStructure}, we will show that $\Loc$ admits a model structure. As with $\PsTop$, we will first recall some basic theory about locales before verifying the assumptions of \cref{Q-axioms}.

\begin{definition}\label{FrameDef}
A \textit{frame} is a complete lattice $(L, \leq)$ which satisfies the infinitary distributive law
\[
a \wedge (\bigvee_{k \in K} x_k) = \bigvee_{k \in K} (a \wedge x_k)\text{.}
\]
Given two frames $L$ and $M$, a \textit{frame homomorphism} is a monotone function $f \from L \to M$ such that $f(\bigvee K) = \bigvee fK$ for all $K \subseteq L$ and $f(\bigwedge_{k=1}^n x_k) = \bigwedge_{k=1}^n f(x_k)$ for all $\{x_k\}_{k=1}^n \subseteq L$.
\end{definition}

The category of frames and frame homomorphisms is denoted $\Frm$.

\begin{definition}
 The category of locales, denoted $\Loc$, is the opposite of the category of frames $\Frm^{op}$.
\end{definition}

Therefore locales are frames, but a locale morphism $L \to M$ is a frame homomorphism $f \from M \to L$. Since frames are complete lattices, there is a right adjoint $g \from L \to M$ going the localic direction; such a map is a \textit{localic map}, which preserves all meets. When necessary, we distinguish between the left and right adjoint by writing $f^L$ and $f^R$. Since a Galois connection $f : L \rightleftarrows M : g$ satisfies $fgf = f$ and $gfg = g$, $f$ is injective \iff\ $g$ is surjective, and $f$ is surjective \iff\ $g$ is injective. 

Let $\Omega \from \Top \to \Loc$ denote the functor taking a space $X$ to its topology regarded as a frame, and a continuous map $f \from X \to Y$ to a frame homomorphism $f^{-1} \from \Omega Y \to \Omega X$. Restricting $\Omega$ to $\Sob$, this becomes a full embedding which admits a right adjoint.

\begin{proposition}\cite{stone:spaces}*{Corollary II.1.7}
There is an adjunction $\Omega : \Sob \rightleftarrows \Loc : \op{pt}$. In this, $\Omega$ is a full embedding of $\Sob$ into $\Loc$, making $\Sob$ a coreflective subcategory of $\Loc$. \qed
\end{proposition}

We will take the restriction of $\Omega$ to $\Box$ to be the inclusion $\iota$

\begin{notation} \label{notation:Iloc_Jloc}
We write $\mathcal{I}_\Loc$ and $\mathcal{J}_\Loc$ for the images of $\mathcal{I}$ and $\mathcal{J}$ under $\Omega$, respectively.
\end{notation}

A locale $L$ is \textit{spatial} if it is in the essential image of $\Omega$; that is, if $L \iso \Omega X$ for some sober space $X$. As a left adjoint, $\Omega$ preserves colimits, and in particular, the pushouts described in the section on the model structure on $\Top$.

For a frame $L$, a \textit{prenucleus} is an order-preserving map $k_0 \from L \to L$ such that $x \leq k_0(x)$ and $k_0(x) \wedge y \leq k_0(x \wedge y)$ for all $x,y \in L$. A \textit{nucleus} is a closure operator $k \from L \to L$ that preserves binary meets. Any prenucleus $k_0$ generates a nucleus with the same fixed points as $k_0$ by letting $\op{Fix}(k_0) \subseteq L$ be the set of fixed points of $k_0$ and defining $k(x) = \bigwedge \set{y \in \op{Fix}(k_0)}{x \leq y}$. We will now fix a construction of the necessary limits and colimits in $\Frm$ and $\Loc$. 

Limits in $\Frm$ are formed as in $\Set$, with the natural pointwise order on products: $(x_k) \leq (y_k)$ \iff\ $x_k \leq y_k$ for all $k$. We use a description of binary coproducts in $\Frm$ as given by Banaschewski in \cite{banaschewski:fixpoint} and Chen in \cite{binary:frames}. For a frame $L$ and $S \subseteq L$, let $\downarrow S = \set{x \in L}{x \leq s \text{ for some }s \in S}$. Define the downset functor $\mathfrak{D} \from \Frm \to \Frm$ by $\mathfrak{D}X = \set{U \subseteq \mathcal{P}X}{U \text{ downward closed}}$ as the set of downsets (ordered by inclusion, with intersection for meets and unions for joins) and $\mathfrak{D}f(U) = \downarrow f(U)$. To distinguish downsets from other sets, we will denote them with $\downarrow S$ from here on. Given two frames $L_1$ and $L_2$, let $L_1 \times L_2$ be their product. Define the following prenuclei on $\mathfrak{D}(L_1 \times L_2)$:
\begin{align*}
\sigma_0(U) &= \set{\bigvee D}{D \subseteq U \text{ updirected}}\text{;}\\
\pi_1(U) &= \set{(\bigvee X, y)}{X \subseteq L_1,\ X \times \{y\} \subseteq U}\text{;}\\
\hat{\pi}_2(U) &= \set{(x, \bigvee Y)}{Y \subseteq L_2\text{ finite, } \{x\} \times Y \subseteq U}\text{.}
\end{align*}
Let $\sigma$ be the associated nucleus of $\sigma_0$, and let $\pi = \sigma \circ \pi_1 \circ \hat{\pi}_2$. Then the set $\op{Fix}(\pi)$ of fixed elements, also called $\pi$-saturated elements, constitutes the coproduct of $L_1$ and $L_2$ in $\Frm$, denoted $L_1 \otimes L_2$. For a downset $\downarrow S$ to be $\pi$-saturated means that, if $x \in L_1$ is such that $\{(x, y_\alpha)\} \subset \downarrow S$, then $(x, \bigvee y_\alpha) \in \downarrow S$; likewise for joins in $L_1$. The coproduct inclusions $\iota_i \from L_i \to L_1 \otimes L_2$ are given by $\iota_1(x) = \set{(a,b)}{a \leq x} \cup \overline{n}$ where $\overline{n} = \set{(a,b)}{a=0\text{ or }b=0}$, likewise for $\iota_2$. The elements of the form $\iota_1(x) \cap \iota_2(y) = \downarrow(x,y) \cup \overline{n}$ are denoted $x \otimes y$. In particular, due to the possibility of an empty index set for $\pi$-saturation, $(x, \bigvee \emptyset) = (x,0) \in \downarrow S$ and similarly $(0, y) \in \downarrow S$ for all $\downarrow S \in L_1 \otimes L_2$, so $\overline{n}$ is the least element of $L_1 \otimes L_2$. 

For arbitrary $\downarrow S \in L_1 \otimes L_2$, we have the identity $\downarrow S = \bigvee_{a \otimes b \leq \downarrow S} a \otimes b = \bigcup_{a \otimes b \leq \downarrow S} a \otimes b$; moreover, for $\{x_\alpha\} \subseteq L_1$ and $y \in L_2$, we have $\bigvee (x_\alpha \otimes y) = (\bigvee x_\alpha) \otimes y$ \cite{locales:textbook}*{Proposition IV.5.2}. Given a frame homomorphism $f \from L_2 \to M$, the map $\id[L_1] \otimes f \from L_1 \otimes L_2 \to L_1 \otimes M$ satisfies 
$$(\id[L_1] \otimes f)(1 \otimes x) = (\id[L_1] \otimes f)\iota_2(x) = \iota_2 \circ f(x) = (1 \otimes f(x))$$
and likewise for $\iota_1$; thus, $(\id[L_1] \otimes f)(a \otimes b) = a \otimes f(b)$ and so $(\id[L_1] \otimes f)(\downarrow S) = \bigvee_{a \otimes b \leq \downarrow S} a \otimes f(b)$. 

Lastly, we may form the pushout of a span $B \xleftarrow{f} A \xrightarrow{g} C$ in $\Loc$ by taking the underlying set $P$ to be the pullback of the left adjoints in $\Frm$, $P = \set{(b,c) \in B \times C}{f^L(b) = g^L(c)}$. In $\Frm$, the projection maps are the obvious projections, so the pushout legs in $\Loc$ are the right adjoints, namely $i_B \from B \to P$ given by $i_B(x) = \bigvee\set{(b,c) \in P}{b = i_B^L(b,c) \leq x}$; analogously for $i_C$. Note that, since pullbacks of injections (surjections) are again injections (surjections) in $\Frm$ since they are in $\Set$, the analogous statement holds for pushouts in $\Loc$.

\begin{proposition}\cite{locales:textbook}*{Corollary IV.4.3.5}
The categories $\Frm$ and $\Loc$ are bicomplete. \qed
\end{proposition}

The coproduct in $\Frm$ distributes over arbitrary products, or equivalently, the product in $\Loc$ distributes over arbitrary coproducts.

\begin{proposition}\cite{joyal:galois}*{Proposition I.5.2} \label{ProductDistributeLocale}
Let $L$ be a frame and $\{M_k\}_{k\in K}$ be a collection of frames. Then $L \otimes \prod_{k\in K} M_k \iso \prod_{k\in K} (L \otimes M_k)$ is a product with projections $\id[L] \otimes p_k$, where $p_k$ are the projections from $\prod_{k\in K} M_k$. \qed
\end{proposition}

\begin{proposition}\label{LocSpatialProducts}\cite{stone:spaces}*{Proposition II.2.13}
If $X$ and $Y$ are sober spaces with $X$ locally compact, then $\Omega(X \times Y) \iso \Omega X \otimes \Omega Y$, where $\otimes$ is the product in $\Loc$. \qed
\end{proposition}

The initial object in $\Frm$ is the poset $T = \{0 < 1\}$: since every frame homomorphism preserves $0$ and $1$, this uniquely determines a morphism out of $T$. Terminal objects in $\Frm$ are trivial poset $*$: the constant map from any frame is clearly a frame homomorphism. Note that, since $0 = 1$ in $*$, there are no maps out of $*$, except to singletons. Considering the dual then, any singleton is an initial object in $\Loc$, and $T$ is the terminal object in $\Loc$. From this, it is clear that $\Omega \emptyset$ is initial and $\Omega *$ is terminal. In particular, since there are no maps into $\Omega \emptyset$ except the identity, any product with $\Omega \emptyset$ is again $\Omega \emptyset$ hence $\Omega \emptyset$ satisfies (S4).

In a frame, $a$ is \textit{way below} $b$, written $a \ll b$, if, whenever $b = \bigvee K$, there is a finite subset $K' \subseteq K$ such that $a \leq \bigvee K'$. A frame is \textit{continuous} or \textit{locally compact} if $a = \bigvee\set{x \in L}{x \ll a}$ for all $a \in L$; a locale is locally compact if it is locally compact as a frame. Note that, if $X$ is a locally compact Hausdorff space, then $\Omega X$ is locally compact. Johnstone characterizes the exponentiable locales as the locally compact locales, as described in the following proposition.

\begin{proposition}\label{LocaleExponentiability}\cite{stone:spaces}*{Theorem VII.4.11}
A locale is exponentiable in $\Loc$ \iff\ it is locally compact. \qed
\end{proposition}

Hence the images of $\bd I^n$ and $I^n$ under $\Omega$ are exponentiable in $\Loc$. Let us now verify that $\mathcal{I}_\Loc$ and $\mathcal{J}_\Loc$ admit the small object argument.

\begin{lemma} \label{lem:transfinite_comp_locales}
Let $X_0 \to X = \colim X_n$ be an $\omega$-transfinite composite of injective localic maps. Denote the colimit inclusions by $i_n \from X_n \ito X$. A map $f \from A \to X$ factors through $X_N$ \iff\ $fA \subseteq i_NX_N$.
\end{lemma}
\begin{proof}
The forward direction is clear. Suppose conversely that $fA \subseteq i_NX_N$ for some positive integer $N$, so for each $a \in A$ there is $x_a \in X_N$ with $f(a)=i_N(x_a)$. Define $\overline{f} \from A \to X_N$ by $\overline{f}(a)=x_a$. Clearly, $i_N\overline{f}(a)=f(a)$, so since $i_N$ is injective and both $f$ and $i_N$ are localic maps, it follows that $\overline{f}$ is a localic map as well by \cite{locales:textbook}*{IV.1.1.1}.
\end{proof}

In the following, we write an element of $\coprod \Omega X_\alpha$ as a union: $\bigcup U_\alpha$ where $U_\alpha$ is an open subset of $X_\alpha$.

\begin{proposition}\label{SOALocales}
Let $K$ be a compact Hausdorff space. Then $\Omega K$ is small relative to transfinite composites of pushouts of coproducts of maps in $\mathcal{I}_\Loc$ and $\mathcal{J}_\Loc$ in $\Loc$.
\end{proposition}
\begin{proof}
Let $X_0 \to X = \colim X_n$ be an $\omega$-transfinite composite of pushouts of coproducts of maps in $\mathcal{K}_\Loc$, where $\mathcal{K}_\Loc = \{\Omega i_n \from \Omega S_n \to \Omega C_n\}_{n \geq 0}$ is either $\mathcal{I}_\Loc$ or $\mathcal{J}_\Loc$. Note that when $X_0 \iso \Omega *$, any map $f \from \Omega K \to X$ factors through some $X_n$ since it does in $\Sob$ and $\Omega$ preserves colimits. We can reduce the general case to this as follows. Let $Y_0 = \Omega *$, and inductively define $Y_{n+1}$ as the following pushout
\[\begin{tikzcd}
	{\coprod \Omega S_{n_\alpha}} & {X_n} & {Y_n} \\
	{\coprod \Omega C_{n_\alpha}} & {X_{n+1}} & {Y_{n+1}}
	\arrow["{p_n}", from=1-2, to=1-3]
	\arrow["\sigma", from=1-1, to=1-2]
	\arrow["\coprod \Omega i_{n_\alpha}"', hook, from=1-1, to=2-1]
	\arrow[from=2-1, to=2-2]
	\arrow["i_X", hook, from=1-2, to=2-2]
	\arrow["\lrcorner"{anchor=center, pos=0.125, rotate=180}, draw=none, from=2-2, to=1-1]
	\arrow["i_Y", hook, from=1-3, to=2-3]
	\arrow["{p_{n+1}}"', dashed, from=2-2, to=2-3]
	\arrow["\lrcorner"{anchor=center, pos=0.125, rotate=180}, draw=none, from=2-3, to=1-1]
\end{tikzcd}\]
where the left square is the defining pushout for $X_{n+1}$, and both the outer and right squares are pushouts. Explicitly, since $(\coprod \Omega i_{n_\alpha})^L(\bigcup U_\alpha) = (\coprod i_{n_\alpha})^{-1}(\bigcup U_\alpha) = \bigcup (U_\alpha \cap S_{n_\alpha})$ for $\bigcup U_\alpha \subseteq \bigcup C_{n_\alpha}$, we obtain the following descriptions of $X_{n+1}$ and $Y_{n+1}$:
\begin{align*}
X_{n+1} = \set{(x, \bigcup U_\alpha) \in X_n \times \coprod \Omega C_{n_\alpha}}{\sigma^L(x) = \bigcup (U_\alpha \cap S_{n_\alpha})}\text{;}\\
Y_{n+1} = \set{(y, x, \bigcup U_\alpha) \in Y_n \times X_{n+1}}{p_n^L(y) = i_X^L(x, \bigcup U_\alpha) = x}\text{.}
\end{align*}
For $x \in X_n$, let $\mathfrak{m}(x) = \sigma^L(x) \cup \bigcup (C_{n_\alpha} - S_{n_\alpha})$ which is open since each $i_{n_\alpha}$ is a closed map. Then $i_X(x) = \bigvee\set{(\tilde{x}, \bigcup V_\alpha) \in X_{n+1}}{\tilde{x} \leq x}$; clearly $(x, \mathfrak{m}(x))$ is the maximum within this set, so $i_X(x) = (x, \mathfrak{m}(x))$. Similarly, we have $i_Y(y) = (y, p_n^L(y), \mathfrak{m}(p_n^L(y)))$.
Repeating this process for all $n$, we get an induced map $\hat{p} \from X \to Y = \colim Y_n$ in the diagram
\[\begin{tikzcd}
	{X_0} & {X_1} & {X_2} & \cdots & {X} \\
	{Y_0} & {Y_1} & {Y_2} & \cdots & {Y}
	\arrow["{p_0}"', from=1-1, to=2-1]
	\arrow["{p_1}"', from=1-2, to=2-2]
	\arrow["{p_2}"', from=1-3, to=2-3]
	\arrow["{\hat{p}}", from=1-5, to=2-5]
	\arrow[hook, from=1-1, to=1-2]
	\arrow[hook, from=1-2, to=1-3]
	\arrow[hook, from=2-1, to=2-2]
	\arrow[hook, from=2-2, to=2-3]
	\arrow[hook, from=1-3, to=1-4]
	\arrow[hook, from=1-4, to=1-5]
	\arrow[hook, from=2-3, to=2-4]
	\arrow[hook, from=2-4, to=2-5]
\end{tikzcd}\]
By \cref{lem:transfinite_comp_locales}, a map $f \from \Omega K \to X$ factors through $X_N$ \iff\ $f(\Omega K)$ is contained in $X_N$, identified with its inclusion into $X$. Suppose there was a map $f \from \Omega K \to X$ that did not factor through any $X_n$. Necessarily then, $f(\Omega K)$ intersects the image of $X_{n+1} - i_XX_n$ for arbitrarily large $n$; we will show that this holds for $Y_n$ as well.

If $(x, \bigcup U_\alpha) \in f(\Omega K) \cap (X_{n+1} - i_XX_n)$, then necessarily $\bigcup U_\alpha \neq \mathfrak{m}(x)$. Since $\sigma^L(x) = \bigcup (U_\alpha \cap S_{n_\alpha})$, we have $\bigcup (U_\alpha \cap S_{n_\alpha}) = \mathfrak{m}(x) \cap \bigcup S_{n_\alpha}$, hence there must be a point $t \in (\mathfrak{m}(x) - U_\alpha) \cap (\bigcup (C_{n_\alpha} - S_{n_\alpha}))$. In particular, $t \in \mathfrak{m}(x')$ for any $x' \in X_n$ since $\bigcup (C_{n_\alpha} - S_{n_\alpha}) \subseteq \mathfrak{m}(x')$. Let 
$$(\tilde{y}, \tilde{x}, \bigcup V_\alpha) = p_{n+1}(x, \bigcup U_\alpha) = \bigvee\set{(y', x', \bigcup V_\alpha') \in Y_{n+1}}{ (x', \bigcup V_\alpha') \leq (x, \bigcup U_\alpha)}\text{.}$$
Then $\bigcup V_\alpha \subseteq \bigcup U_\alpha$, so $t \notin V_\alpha$. But by the description of $i_Y$, if $(\tilde{y}, \tilde{x}, \bigcup V_\alpha) = i_Y(y)$ for some $y \in Y_n$, then $\bigcup V_\alpha = \mathfrak{m}(p_n^L(y))$, which cannot be true since $t \in \mathfrak{m}(p_n^L(y))$. Thus, $p_{n+1}(x, \bigcup U_\alpha) \notin i_Y Y_n$.

Hence $\hat{p}f(\Omega K)$ intersects $Y_{n+1} - i_YY_n$ for arbitrarily large $n$ too, which is a contradiction since $\hat{p}f$ must factor through some $Y_N$. Thus, no such $f$ exists, so all maps $f \from \Omega K \to X$ factor through some $X_n$.
\end{proof}

The only remaining condition to check is that products with $L$ preserve the pushout in (S3) for any locale $L$, which we will verify the dual statement for in $\Frm$.
This however requires a few lemmas. Recall that, as in the beginning of \cref{Ex:topological_spaces}, $a(t) = t/2$ and $b(t) = (t+1)/2$ are the pushout legs $I \to I$ such that $a$ restricts to a homeomorphism $[0, 1) \to [0, 1/2)$ and $b$ restricts to a homeomorphism $(0, 1] \to (1/2, 1]$.

\begin{lemma}\label{BinCodFrm}\cite{binary:frames}*{Proposition 2.2}
If $N$ is a continuous frame and $a \otimes b \leq \pi(U)$ in $L \otimes N$, then for any $c \ll b$ we have $(a, c) \in \pi_1\hat{\pi}_2(U)$. \qed
\end{lemma}

\begin{lemma}\label{InjectivityLemma}
Let $L$ be a frame and $\downarrow S \in L \otimes \Omega I$. Then the following holds:
\begin{enumerate}
\item[(1)] $\mathfrak{D}(\id[L] \times a^{-1}) (\downarrow S) = (\id[L] \times a^{-1})_*(\downarrow S) = \set{(x,a^{-1}U)}{(x,U) \in \downarrow{S})}$. That is, the image of $\downarrow S$ under $\id[L] \times a^{-1}$ is a downset.
\item[(2)] $(\id[L] \times a^{-1})_*(\downarrow S) = \hat{\pi}_2 ((\id[L] \times a^{-1})_*(\downarrow S)) = \pi_1 ((\id[L] \times a^{-1})_*(\downarrow S))$. In particular, $(\id[L] \times a^{-1})_*(\downarrow S)$ is fixed under $\pi_1\hat{\pi}_2$.
\item[(3)] $\id[L] \otimes a^{-1}(\downarrow S) \leq \pi \circ (\id[L] \times a^{-1})_*(\downarrow S)$.
\end{enumerate}
Analogous statements hold for $b^{-1}$ in place of $a^{-1}$.
\end{lemma}
\begin{proof}
Since $\downarrow S \in L \otimes \Omega I$, it is fixed under $\pi$, hence if $\{(x, U_\alpha)\} \subseteq \downarrow S$ then $(x, \bigcup U_\alpha) \in \downarrow S$. For (1) then, suppose $(x, U) \in \mathfrak{D}(\id[L] \times a^{-1}) (\downarrow S) = \downarrow [(\id[L] \times a^{-1})(\downarrow(S))]$. Then there is $(y, V) \in \downarrow S$ such that $(x,U) \leq (y, a^{-1}V)$, so $x \leq y$ and $U \subseteq a^{-1}V$. If $1 \notin U$, then $aU$ is open, $aU \subseteq aa^{-1}V \subseteq V$, and $a^{-1}aU = U$, so $(x, aU) \in \downarrow S$ satisfies $(\id[L] \times a^{-1})(x, aU) = (x,U)$. Otherwise, if $1 \in U$ and hence $1/2 \in V$, then $W = aU \cup ([1/2,1] \cap V)$ is open and satisfies $a^{-1}W = U$ and $W \subseteq V$, so $(x, W) \in \downarrow S$ satisfies $(\id[L] \times a^{-1})(x, W) = (x, U)$. Thus, the direction $\subseteq$ holds, and the other inclusion is clear, so they are equal.\\
For the remainder of the proof, let $\cat{S} = (\id[L] \times a^{-1})_*(\downarrow S)$. For (2), suppose first that $(x, V) \in \hat{\pi}_2 \cat{S}$. Then there must be some finite subset $\{a^{-1}U_n\}_{n=1}^N \subseteq \Omega I$ such that $\{x\} \times \{a^{-1}U_n\}_{n=1}^N \subseteq \cat{S}$ and $\bigcup_{n=1}^N a^{-1}U_n = V$. Since each $(x,a^{-1}U_n) \in \cat{S}$, $(x, U_n) \in \downarrow S$, so since $\downarrow S$ is $\pi$-saturated $(x, \bigcup_{n=1}^N U_n) \in \downarrow S$ satisfies $(x,V) = (x, a^{-1}(\bigcup_{n=1}^N U_n)) \in \cat{S}$. Thus, $\hat{\pi}_2 \cat{S} \subseteq \cat{S}$. Similarly, if $(\bigvee x_\alpha, a^{-1}U) \in \pi_1 \cat{S}$ for some $\{(x_\alpha, a^{-1}U)\} \subseteq \cat{S}$, then $(x_\alpha, U) \in \downarrow{S}$ for each $\alpha$, hence by $\pi$-saturation we have $(\bigvee x_\alpha, U) \in \downarrow S$ and $(\bigvee x_\alpha, a^{-1}U) \in \cat{S}$. Thus, $\pi_1 \cat{S} \subseteq \cat{S}$. Since the reverse inclusions are clear, the equalities hold.\\
Lastly, for (3), for each $(x, U) \in \downarrow S$ we have $(x,a^{-1}U) \in \cat{S}$, hence in $\pi \cat{S}$, so $\id[L] \otimes a^{-1}(x, U) = x \otimes a^{-1}U \leq \pi \cat{S}$. Since this holds for all such $(x, U)$, we have $$\id[L] \otimes a^{-1} (\downarrow S) = \bigvee_{x \otimes U \leq \downarrow S} x \otimes a^{-1}U \leq \pi \cat{S}\text{.}$$
\end{proof}

\begin{lemma}\label{FrameDensityLemma}
Let $f \from M \to N$ be a frame homomorphism, and suppose $\{n_k\}_{k \in K} \subseteq N$ is such that every $n \in N$ is of the form $n = \bigvee_{k \in K'} n_k$ for some $K' \subseteq K$. If $\{n_k\}_{k\in K}$ is in the image of $f$, then $f$ is surjective.
\end{lemma}
\begin{proof}
Suppose that $f(m_k)=n_k$ for each $k \in K$. Fix $n \in N$, and let $K' \subseteq K$ be such that $n = \bigvee_{k \in K'}n_k$. Then $\bigvee_{k\in K'}m_k$ satisfies $f(\bigvee_{k\in K'}m_k) = \bigvee_{k\in K'}f(m_k)=\bigvee_{k\in K'} n_k = n$.
\end{proof}

We will identify $L \otimes \{0 < 1\}$ with $L$ and write $\chi_k$ for $\id[L] \otimes i_k \from L \otimes \Omega I \to L$, for $k =0,1$. For $x \in L$, $U \in \Omega I$, we have $\chi_k(x \otimes U) = x$ if $k \in U$ and $\chi_k(x \otimes U) = 0$ if $k \notin U$; for a general $\downarrow S \in L \otimes \Omega I$, we have $\chi_k(\downarrow S) = \bigvee \set{x \in L}{\exists U \in \Omega I\text{ with } k \in U \text{ and } x \otimes U \leq \downarrow S}$.

\begin{proposition}\label{FrmPullbackLemma}
For any frame $L$, the following is a pullback square in $\Frm$:
\[\begin{tikzcd}
	{L \otimes \Omega I} && {L \otimes \Omega I} \\
	{L \otimes \Omega I} && L
	\arrow["{\chi_0}"', from=2-1, to=2-3]
	\arrow["{\chi_1}", from=1-3, to=2-3]
	\arrow["{\id[L] \otimes b^{-1}}"', from=1-1, to=2-1]
	\arrow["{\id[L] \otimes a^{-1}}", from=1-1, to=1-3]
\end{tikzcd}\]
\end{proposition}
\begin{proof}Let $P$ be the canonical pullback of the square (constructed as a subset of the product); since limits in $\Frm$ have the corresponding limit in $\Set$ for their underlying set, we may take 
$$P = \set{(\downarrow S, \downarrow T) \subseteq (L \otimes \Omega I)^2 }{\chi_1(\downarrow S) = \chi_0(\downarrow T)}$$
along with the usual projections $\pr_1$ and $\pr_2$. Let $\phi \from L \otimes \Omega I \to P$ be the induced map, which is explicitly the map $\phi(\downarrow S) = (\id[L] \otimes a^{-1}(\downarrow S), \id[L] \otimes b^{-1}(\downarrow{S}))$. We will show that $\phi$ is a bijection, hence an isomorphism.\\
By \cref{FrameDensityLemma}, it suffices to show every $(x \otimes U, y \otimes V)$ is in the image of $\phi$ for surjectivity, since every element of $P$ can be written as a join of such elements. Indeed, let $(\downarrow S, \downarrow T) \in P$. For each $x \otimes U \leq \downarrow S$, if $1 \notin U$ then let $t_{x \otimes U} = 0$ so $(x \otimes U, t_{x \otimes U}) \in P$. Otherwise, if $1 \in U$ then let $\mathscr{T} = \set{y \otimes V \leq \downarrow T}{0 \in V}$, so for each $y \otimes V \in \mathscr{T}$, we have that $((x \wedge y) \otimes U, (x \wedge y) \otimes V) \in P$. Since $\chi_0(\downarrow T) = \chi_1(\downarrow S)$ and this holds for all such $y \otimes V$, we have
\begin{align*}
\bigvee_{y \otimes V \in \mathscr{T}} ((x \wedge y) \otimes U, (x \wedge y) \otimes V) =& \bigg( \bigg(\bigvee_{y \otimes V \in \mathscr{T}} (x \wedge y)\bigg) \otimes U, t_{x \otimes U}\bigg)\\
=&\bigg(\bigg(x \wedge \bigvee_{y \otimes V \in \mathscr{T}} y\bigg) \otimes U, t_{x \otimes U}\bigg)\\
=& ((x \wedge \chi_1(\downarrow S)) \otimes U, t_{x \otimes U}) = (x \otimes U, t_{x \otimes U}) \leq (\downarrow S, \downarrow T)
\end{align*}
where $t_{x \otimes U} = \bigvee_{y \otimes V \in \mathscr{T}}((x \wedge y) \otimes V)$; the first equality holds by $\bigvee (x_\alpha \otimes y) = (\bigvee x_\alpha) \otimes y$, the second by distributivity, and the last holds since $x \leq \chi_1(\downarrow S)$. Thus, whether $1 \in U$ or not, we have that $(x \otimes U, t_{x\otimes U})$ can be written as a join of pairs of tensor elements, hence so can
$$
\bigvee_{x \otimes U \leq \downarrow S} (x \otimes U, t_{x\otimes U}) = \bigg(\bigvee_{x \otimes U \leq \downarrow S} x \otimes U, t_{\downarrow S}\bigg) = (\downarrow S, t_{\downarrow S}) \in P\text{,}
$$
where $t_{\downarrow S} = \bigvee_{x \otimes U \leq \downarrow S} t_{x\otimes U} \leq \downarrow T$. Analogous reasoning for $\downarrow T$ shows that \\$(\downarrow S, \downarrow T) = (\downarrow S, t_{\downarrow S}) \vee (s_{\downarrow T}, \downarrow T)$ can be written as a join of elements of the form $(x \otimes U, y \otimes V)$.

For surjectivity then, fix an element $(x \otimes U, y\otimes V) \in P$. Suppose first that $1 \in U$. If $x = 0$, then $\chi_1(x \otimes U) = 0 = \chi_0(y \otimes V)$; thus, either $0 \notin V$ or $y = 0$. If $y = 0$, then $\phi(0)=(0,0)=(x \otimes U, y \otimes V)$, so suppose not. Then since $0 \notin V$, $bV$ is open in $I$ and $a^{-1}bV$ is empty, so $y \otimes bV$ satisfies 
$$\phi(y \otimes bV) = (y \otimes a^{-1}bV, y \otimes b^{-1}bV) = (0, y \otimes V) = (x \otimes U, y \otimes V)\text{.}$$
Suppose now that $x \neq 0$. Then $x = \chi_1(x \otimes U) = \chi_0(y \otimes V)$; necessarily then $0 \in V$ and $y = x$. Then $C = aU \cup bV$ is open in $I$ and satisfies $a^{-1}C = U$ and $b^{-1}C = V$, so $\phi(x \otimes C) = (x \otimes U, y \otimes V)$. Suppose now that $1 \notin U$, so $aU$ is open in $I$ and $b^{-1}aU$ is empty. If $y = 0$, then $\phi(x \otimes aU) = (x \otimes a^{-1}aU, 0) = (x\otimes U, y \otimes V)$. If $y \neq 0$ then since $\chi_0(y \otimes V) = \chi_1(x \otimes U) = 0$, it must be that $0 \notin V$. Thus, $aU$ and $bV$ are both open in $I$ and $a^{-1}bV$ is empty, so
$$
\phi((x \otimes aU) \vee (y \otimes bV)) = \phi(x \otimes aU) \vee \phi(y \otimes bV) = (x \otimes U, 0) \vee (0, y \otimes V) = (x \otimes U, y \otimes V)\text{.}
$$
Thus, in any case we see that $(x \otimes U, y \otimes V)$ is in the image of $\phi$, hence $\phi$ is surjective.

For injectivity, let $\downarrow S \neq \downarrow T$ in $L \otimes \Omega I$. Without loss of generality, assume there is $(x, U) \in \downarrow S - \downarrow T$. Suppose that $(x \otimes a^{-1}U,x \otimes b^{-1}U) \leq \phi(\downarrow T)$. Then we have that $x \otimes a^{-1}U \leq (\id[L]\otimes a^{-1})(\downarrow T) \leq \pi \circ (\id[L] \times a^{-1})_*(\downarrow T)$. For each $t_a \in a^{-1}U$, by local compactness we may choose an open interval $I_{t_a}$ (half open interval $[0, \alpha)$ or $(\beta, 1]$ if $t_a = 0,1$ is in $a^{-1}U$), not containing $1$ unless $t_a=1$, such that $I_{t_a} \ll a^{-1}U$. Thus, by \cref{BinCodFrm} and \cref{InjectivityLemma}, $(x, I_{t_a}) \in (\id[L] \times a^{-1})_*(\downarrow T)$. Similarly, for each $t_b \in b^{-1}U$, there is an interval $J_{t_b}$, not containing $0$ unless $t_b=0$, such that $(x, J_{t_b}) \in (\id[L] \times a^{-1})_*(\downarrow T)$. If $1 \in a^{-1}U$, and hence $0 \in b^{-1}U$, then we may choose some neighborhood $J_{1/2} \subseteq U$ of $1/2$ such that $(x, J_{1/2}) \in \downarrow T$, $a^{-1}J_{1/2}=I_{1}$ and $b^{-1} J_{1/2} = J_0$; otherwise let $J_{1/2} = \emptyset$. For all $t_a \neq 1$ in $a^{-1}U$, we have that $aI_{t_a} \subseteq U$ is open since $1 \notin I_{t_a}$, and $(x, aI_{t_a}) \in \downarrow T$ since $aI_{t_a}$ is the least open set $V$ such that $a^{-1}V =I_{t_a}$; likewise for each $t_b \neq 0$, $bJ_{t_b} \subseteq U$ is open and $(x, bJ_{t_b}) \in \downarrow T$. Thus, since $a$ and $b$ are jointly surjective, each $t \in U$ is in some $aI_{t_a}$, $bJ_{t_b}$, or $J_{1/2}$, so their union is $U$. Since $\downarrow T$ is $\pi$-saturated, we have that
$$
(x, \bigvee_{t_a \in a^{-1}U} aI_{t_a} \bigvee_{t_b \in b^{-1}U} bJ_{t_b} \vee J_{1/2}) = (x, U) \in \downarrow T
$$
which contradicts our assumption. Thus, $(x \otimes a^{-1}U,x \otimes b^{-1}U) \nleq \phi(\downarrow T)$ but $(x \otimes a^{-1}U,x \otimes b^{-1}U) \leq \phi(\downarrow S)$ and hence $\phi(\downarrow T) \neq \phi(\downarrow S)$.
\end{proof}

We are now ready to verify that $(\Loc, \Omega)$ carries a model structure. 

\begin{theorem}
There is a model structure on $\Loc$ whose cofibrations are maps in $cof(\mathcal{I}_\Loc)$ (cf.~\cref{notation:Iloc_Jloc}), fibrations are maps in $rlp(\mathcal{J}_\Loc)$ (cf.~idem), and weak equivalences are maps which have the weak right lifting property against $\mathcal{I}_\Loc$, or equivalently by $\cref{WEClassicDef}$, maps which induce isomorphisms on all homotopy groups.
\end{theorem}

\begin{proof}
   We verify that $(\Loc, \Omega)$ is a good Q-structure.
    By the description of $\Omega$, (S1) and (S4) are clear, (S2) and (Q7) hold by \cref{LocSpatialProducts}, and (S3) holds by \cref{FrmPullbackLemma}; thus, $(\Loc, \Omega)$ is a category with intervals. By \cref{ProductDistributeLocale} (Q2) holds, (Q5) follows from \cref{LocaleExponentiability}, and (Q6) by \cref{SOALocales}. Lastly, (Q3), (Q4), and (Q8) holds by $\Omega$ preserving colimits. 
\end{proof}

Since $\Omega$ preserves colimits and the generating (acyclic) cofibrations, we have an immediate corollary.

\begin{proposition}
The adjunction $\Omega \adj \op{pt}$ is a Quillen adjunction between the Quillen model structures on $\Sob$ and $\Loc$.
\end{proposition}

We will show that $\Omega \adj \op{pt}$ is in fact a Quillen equivalence. Note that every object in $\Loc$ is fibrant, and for any $X \in \Sob$, $X \to \op{pt} \Omega X$ is an isomorphism as $\Omega$ is fully faithful. Thus, by \cite{hovey:book}*{Corollary 1.3.16(c)} it is necessary and sufficient to show that $\op{pt}$ reflects weak equivalences. The following lemma will be useful.

\begin{lemma}\label{SobLocHomotopies}
Let $i \from A \to X$ be a map between sober spaces and $f,g \from \Omega X \to L$ be maps in $\Loc$. Then $f \sim g$ rel $\Omega A$ in $\Loc$ if and only if $\hat{f} \sim \hat{g}$ rel $A$ in $\Sob$, where $\hat{f}, \hat{g} \from X \to \op{pt}L$ are the adjunct maps.
\end{lemma}
\begin{proof}
Recall from \cref{LocSpatialProducts} that the canonical map $\Omega(X \times I) \to \Omega X \otimes \Omega I$ is an isomorphism as $I$ is locally compact, where $\otimes$ is again the product in $\Loc$. By this and adjointness, we have that a map $K \from X \times I \to \op{pt}L$ makes the following two diagrams commute in $\Sob$
\[\begin{tikzcd}
	X & {X \times I} & X && {A \times I} & A \\
	& {\op{pt}L} &&& {X \times I} & {\op{pt} L}
	\arrow["K"{description}, from=1-2, to=2-2]
	\arrow["{e_0}", from=1-1, to=1-2]
	\arrow["{e_1}"', from=1-3, to=1-2]
	\arrow["{\hat{f}}"', from=1-1, to=2-2]
	\arrow["{\hat{g}}", from=1-3, to=2-2]
	\arrow["{\hat{f}i = \hat{g}i}", from=1-6, to=2-6]
	\arrow["{\pi_{A}}", from=1-5, to=1-6]
	\arrow["{i \times \id[I]}"', from=1-5, to=2-5]
	\arrow["K"', from=2-5, to=2-6]
\end{tikzcd}\]
if and only if $\Omega X \otimes \Omega I \iso \Omega(X \times I)  \xrightarrow{\hat{K}} L$ makes the following two diagrams commute in $\Loc$
\[\begin{tikzcd}
	{\Omega X} & {\Omega X \otimes \Omega I} & {\Omega X} && {\Omega A \otimes \Omega I} & {\Omega A} \\
	& L &&& {\Omega X \otimes \Omega I} & L
	\arrow["{\hat{K}}"{description}, from=1-2, to=2-2]
	\arrow["{e_0}", from=1-1, to=1-2]
	\arrow["{e_1}"', from=1-3, to=1-2]
	\arrow["f"', from=1-1, to=2-2]
	\arrow["g", from=1-3, to=2-2]
	\arrow["{f \Omega i = g\Omega i}", from=1-6, to=2-6]
	\arrow["{\pi_{\Omega A}}", from=1-5, to=1-6]
	\arrow["{\Omega i \otimes \id[\Omega I]}"', from=1-5, to=2-5]
	\arrow["{\hat{K}}"', from=2-5, to=2-6]
\end{tikzcd}\]
\end{proof}

\begin{theorem}
The adjunction $\Omega \adj \op{pt}$ is a Quillen equivalence between the Quillen model structures on $\Sob$ and $\Loc$.
\end{theorem}
\begin{proof}
Suppose that $f \from X \to Y \in \Loc$ is a map such that $\op{pt}f$ is a weak equivalence. Given the square on the left, we obtain the square on the right by adjointness:
\[\begin{tikzcd}
	{\Omega \bd I^n} & X && {\bd I^n} & {\op{pt}X} \\
	{\Omega I^n} & Y && {I^n} & {\op{pt}Y}
	\arrow["{\Omega i_n}"', hook, from=1-1, to=2-1]
	\arrow["f", from=1-2, to=2-2]
	\arrow["u", from=1-1, to=1-2]
	\arrow["v"', from=2-1, to=2-2]
	\arrow["{\op{pt}f}", from=1-5, to=2-5]
	\arrow["{\hat{u}}", from=1-4, to=1-5]
	\arrow["{\hat{v}}"', from=2-4, to=2-5]
	\arrow["{i_n}"', hook, from=1-4, to=2-4]
\end{tikzcd}\]
As $\op{pt}f$ is a weak equivalence, we obtain a map $h \from I^n \to \op{pt}X$ with $hi_n = \hat{u}$ and a homotopy $K \from \op{pt}f \circ h \sim \hat{v}$ rel $\bd I^n$. 
Then by adjointness, $\hat{h} \circ \Omega i_n = u$, and by \cref{SobLocHomotopies}, $\hat{K}$ is a homotopy $f \hat{h} \sim v$ rel $\Omega (\bd I^n)$. Thus, $\hat{h}$ is the required filler in the original diagram, hence $f$ is a weak equivalence.
\end{proof}

In fact, recognizing that $rlp(\mathcal{I}_\Loc) = rlp(\Omega(\mathcal{I}_\Sob)) = \op{pt}^{-1}(rlp(\mathcal{I}_\Sob))$ by adjunction, we have that the weak equivalences and fibrations in $\Loc$ are preserved and reflected by $\op{pt}$. Thus, the model structure on $\Loc$ is precisely the transferred model structure from $\Sob$.

\begin{corollary}
The Quillen model structure on $\Loc$ is the right transferred model structure from $\Sob$ along $\Omega \adj \op{pt}$. \qed
\end{corollary}

With this Quillen equivalence, we can also compare localic homotopy groups to their homotopy groups under $\op{pt}$. Since every locale is fibrant, \cite{hovey:book}*{Corollary 1.3.16(b)} implies that the counit $\Omega \op{pt}L \to L$ is a weak equivalence for every $L \in \Loc$. Thus, we get the following:

\begin{proposition}
If $(X, x_0)$ is a pointed sober space, then $\pi_n^\Sob(X, x_0) \iso \pi_n^\Loc(\Omega X, \Omega x_0)$ for any $n \geq 0$. In particular, for any pointed locale $(L, l_0)$ and $n \geq 0$, $\pi_n^\Sob(\op{pt}L, \op{pt}l_0) \iso \pi_n^\Loc(L, l_0)$ since $\Omega \op{pt} L \to L$ is a weak equivalence. 
\end{proposition}
\begin{proof}
This follows from \cref{SobLocHomotopies}, as two maps $f,g \from I^n \to X \iso \op{pt} \Omega X$ are homotopic rel $\bd I^n$ if and only if $\hat{f}, \hat{g} \from \Omega I^n \to \Omega X$ are homotopic rel $\Omega(\bd I^n)$.
\end{proof}

\bibliographystyle{amsalphaurlmod}
\bibliography{all-refs.bib}

\end{document}